\documentclass[12pt]{article}
\usepackage{latexsym,amsfonts,amsthm,amsmath,amscd,amssymb}
\usepackage[dvips]{graphicx}

\newtheorem{lemma}{Lemma}[section]
\newtheorem{thm}[lemma]{Theorem}
\newtheorem{rem}[lemma]{Remark}
\newtheorem{prop}[lemma]{Proposition}
\newtheorem{cor}[lemma]{Corollary}

\newtheorem{oss}[lemma]{Observation}

\newtheorem{defn}[lemma]{Definition}

\newcommand{\finedimo}{{\hfill\hbox{$\square$}\vspace{2pt}}}

\renewcommand{\hbar}{{\overline{h}}}

\newfont{\Got}{eufm10 scaled 1200}

\newcommand{\mycap} [1] {\caption{\footnotesize{#1}}}

\newcommand{\Int}{\mathop{\rm Int}\nolimits}

\begin{document}

\title{Generalized Mom-structures and ideal triangulations of
$3$-manifolds with non-spherical boundary}

\author{Ekaterina~{\textsc Pervova}\footnote{Partially supported by
RFFI (proj. 10-01-96035) and by the Program ``Mathematical problems
of algebra, topology, function approximation theory, and
applications'' carried out jointly by the Institute for Mathematics
and Mechanics UrO RAN and the Institute for Mathematics SO RAN.}}

\maketitle

\begin{abstract}
\noindent The so-called Mom-structures on hyperbolic cusped
$3$-manifolds without boundary were introduced by Gabai, Meyerhoff,
and Milley, and used by them to identify the smallest closed
hyperbolic manifold. In this work we extend the notion of a
Mom-structure to include the case of $3$-manifolds with non-empty
boundary that does not have spherical components. We then describe a
certain relation between such generalized Mom-structures, called
\emph{protoMom-structures}, internal on a fixed 3-manifold $N$, and
ideal triangulations of $N$; in addition, in the case of non-closed
hyperbolic manifolds without annular cusps, we describe how an
internal geometric protoMom-structure can be constructed starting
from the Epstein-Penner or Kojima decomposition. Finally, we exhibit
a set of combinatorial moves that relate any two internal
protoMom-structures on a fixed $N$ to each other.

(secondary).
\end{abstract}

\section*{Introduction}

This paper is devoted to an extension of the notion of
\emph{Mom-structure}, that is defined for manifolds with toric
boundary, to the case of manifolds with arbitrary non-spherical
non-empty boundary. Mom-structures were introduced by Gabai,
Meyerhoff, and Milley~\cite{Mom1,Mom2} in the context of hyperbolic
$3$-manifolds and served as a tool for them to identify the smallest
closed hyperbolic manifold~\cite{Mom2, Mom3}; we also note that
certain components of a similar notion for compact hyperbolic
manifolds with non-empty totally geodesic boundary were considered
by Kojima-Miyamoto~\cite{KoMi} and
Deblois-Shalen~\cite{DebloisShalen}, still as an instrument for
studying the volume.

A Mom-structure~\cite{Mom1} itself (see Section~\ref{prelim:sec}
below for the precise definition) is a triple of form
$(M,T,\Delta)$, where $M$ is a $3$-manifold with toric boundary, $T$
is a selected torus in $\partial M$, and $\Delta$ is a particular
type of handle decomposition of $M$, that includes a collar of $T$
serving as the ``base'' for the subsequent gluings of $1$- and
$2$-handles. The crucial point is then to consider a fixed
hyperbolic manifold $N$ and the set of Mom-structures
\emph{internal} on $N$ (roughly speaking, the structure is internal
on $N$ if $N$ can be reconstructed from the structure via some Dehn
fillings). The relation with the hyperbolic volume, as suggested
in~\cite{Mom1} and then realized in~\cite{Mom2}, is, very generally,
the following one. Given a complete finite-volume cusped hyperbolic
manifold $N$, consider a horotorus neighbourhood of some cusp. After
expanding it in normal direction, it will eventually encounter
itself, which gives rise to a $1$-handle, while further expansions
create other $1$- and $2$-handles. The fact that for $N$ of low
enough volume this process allows one to construct an internal
Mom-structure on $N$ with at most three $2$-handles (and in fact
other properties) is a crucial result of~\cite{Mom2}.

Considering examples of hyperbolic $3$-manifolds with low Matveev
complexity (available in the census~\cite{www:CP}) that admit an
ideal triangulation into hyperbolic tetrahedra, suggests that
something similar may occur for other types of hyperbolic
$3$-manifolds, and particularly for the ``mixed'' ones (those with
both toric cusps and totally geodesic boundary). This fact served as
a motivation for us to consider a slightly generalized notion of a
Mom-structure, that we call a \emph{protoMom-structure}, and to
study its relations with ideal triangulations; we immediately note
that we do this from a combinatorial point of view rather than from
a geometric one.

The generalization itself is entirely straightforward; essentially,
it consists in replacing $T$ with an arbitrary non-empty surface
$\Sigma$ without spherical components. Our first, and perhaps not
unexpected, result is that every such generalized Mom-structures
internal on a fixed $3$-manifold $N$ with boundary $\Sigma$, can
actually be obtained by thickening some of the arcs and faces of a
suitably chosen ideal triangulation of $N$
(Theorem~\ref{int:protoMom:2:trn:thm}). We emphasize that our proof
of this fact has a constructive nature: how the desired
triangulation can be constructed via a sequence of a small number of
specific moves.

It is the construction just mentioned, rather than just the
existence itself of the triangulation, that enables us to address
the main point of the paper, namely, the determination of a set of
combinatorial moves that relate any two protoMom-structures
(internal on the same manifold) to each other. The definition of the
moves given in Section~\ref{moves:main:sec} is rather natural in
view of Theorem~\ref{int:protoMom:2:trn:thm} and of the easy fact
that an ideal triangulation naturally gives rise to a variety of
internal protoMom-structures (Section~\ref{trn:2:protomom:subsec}).
Indeed, protoMom-structures arise from a triangulation by discarding
some of its faces and by thickening the remaining ones, and by
possibly cancelling some $1$-handles (that are thickenings of the
edges) of valence $1$ with the respective $2$-handles. The way in
which one can change the choice of which faces will be discarded
naturally translates into certain \emph{M-moves}, while the
collapses of $1$-handles translate into so-called \emph{C-moves}. It
is then the content of Theorem~\ref{main:thm} that these moves are
indeed sufficient to relate between them any two protoMom-structures
internal on a fixed manifold.

The structure of the paper is as follows. Section~\ref{prelim:sec}
contains the necessary definitions, while
Section~\ref{trn:e:mom:sec} describes the relations between ideal
triangulations and protoMom-structures.
Section~\ref{mom:subgraph:sec} is devoted to establishing some
preliminary results needed to prove Theorem~\ref{main:thm}, and
Section~\ref{moves:main:sec} contains the definitions of the moves
together with the proof of the main theorem. Throughout the paper we
will employ the piecewise linear viewpoint, which is equivalent to
the smooth one in dimensions $2$ and $3$.

\paragraph{Acknowledgments} This work was carried out while the
author was holding a Junior Visiting position at the \emph{Centro De
Giorgi} of the Scuola Normale Superiore of Pisa, whose staff the
author would like to thank for the excellent working conditions.
Mathematical discussions with other postdocs at the \emph{Centro},
particularly with Pablo Alvarez-Caudevilla, Antoine Lemenant, and
Laura Spinolo, were of indirect, but notable, benefit for the
completion of this paper. The author would also like to thank the
members of the research group in geometry at the Department of
Mathematics of the University of Pisa for their interest, in
particular Roberto Frigerio for the useful suggestions provided and
Carlo Petronio for the very detailed comments on the exposition
style of the first draft of this paper.

\section{Preliminaries and main definitions}\label{prelim:sec}

In this section we recall some known definitions and facts that will
be used in the rest of the paper, and we describe the main object of
our study, which is a generalization of the key definition of
\cite{Mom1} recalled below.

\paragraph{Mom-structures} Given a handle decomposition of some
cobordism, we will call \emph{valence} of a $1$-handle the number of
$2$-handles incident to it (with multiplicity), and \emph{valence}
of a $2$-handle the number of $1$-handles to which it is incident
(with multiplicity). The original definitions due to Gabai,
Meyerhoff, and Milley are as follows.

\begin{defn}\emph{\cite[Definition 0.4]{Mom1}}
A Mom-$n$ structure is a triple $(M,T,\Delta)$ where
\begin{itemize}
  \item $M$ is a compact connected $3$-manifold such that $\partial M$
  is a union of tori;
  \item $T$ is a preferred boundary component of $M$;
  \item $\Delta$ is a decomposition of $M$ obtained as follows:
  \begin{itemize}
    \item take $T\times[0,1]\subset M$ such that $T\times\{0\}=T$;
    \item attach $n$ $1$-handles to $T\times\{1\}$;
    \item add $n$ $2$-handles of valence $3$ in such a way that each $1$-handle has valence at
    least $2$.
  \end{itemize}
\end{itemize}
\end{defn}

\begin{defn}
Let $M$ be a compact connected $3$-manifold, and let
$S\subset\partial M$ be a compact surface, which can be disconnected
or empty. A \emph{general-based handle structure} $\Delta$ on
$(M,S)$ is a decomposition of $M$ obtained in the following way:
 \begin{itemize}
    \item take $S\times[0,1]\subset M$ such that $S\times\{0\}=S$;
    \item add several $0$-handles;
    \item attach finitely many $1$- and $2$-handles to $S\times\{1\}$ and to the
    $0$-handles;
    \item if needed, fill in some spherical boundary components with
    $3$-handles.
 \end{itemize}
\end{defn}

Thus, if $(M,T,\Delta)$ is a Mom-structure then $\Delta$ is a
(particular type of a) general-based handle decomposition on
$(M,T)$.

Given a general-based handle decomposition $\Delta$ on $(M,S)$, we
will call:
\begin{itemize}
  \item \emph{islands} the connected components of the intersection
  of the $1$-handles of $\Delta$ with $\left(S\times\{1\}\right)\cup
\{0-\mbox{handles}\}$;
  \item \emph{bridges} the connected components of the intersection
  of the $2$-handles of $\Delta$ with $\left(S\times\{1\}\right)\cup
\{0-\mbox{handles}\}$;
  \item \emph{lakes} the connected components of the complement in
$\left(S\times\{1\}\right)\cup \{\mbox{boundaries of
}0-\mbox{handles}\}$ of the union of all islands and bridges.
\end{itemize}

\begin{defn}
Let $\Delta$ be a general-based handle structure on $(M,S)$. Then
$\Delta$ is \emph{full} if all the lakes are discs.
\end{defn}

\paragraph{ProtoMom-structures and weak protoMom-structures} The
main object of our study throughout the paper will be the following
specific type of a general-based handle decomposition, that is
obtained by a (rather straightforward) generalization of the notion
of a Mom-structure.

\begin{defn}\label{protoMom:def}
Let $M$ be a compact connected orientable $3$-manifold such that
$\partial M$ can be written as the union of a non-empty surface
$\Sigma$ without spherical components, and some tori. Let $\Delta$
be a general-based handle decomposition on $(M,\Sigma)$ such that
$\Delta$ does not contain $0$- or $3$-handles, each $1$-handle has
valence at least $2$, and each $2$-handle has valence precisely $3$.
Then the triple $(M,\Sigma,\Delta)$ is called a
\emph{protoMom-structure}.
\end{defn}

\begin{rem}
\emph{Let $\Delta$ be a protoMom-structure on $M$ with $\partial M$
the union of a surface $\Sigma$ of the above type and some tori.
Suppose that $\Delta$ contains $n$ $2$-handles; then $\Delta$
contains $n-g+1$ $1$-handles, where $g$ is the genus of $\Sigma$
(or, if $\Sigma$ is disconnected, the sum of the genera of its
connected components).}
\end{rem}

\begin{rem}
\emph{In the case where $\Sigma$ is a torus, a protoMom structure is
simply a Mom structure in the sense of~\cite{Mom1} (see also
\cite{Mom2}).}
\end{rem}

\begin{defn}
Let $\Delta$ be a general-based handle decomposition on $(M,\Sigma)$
with $M$ and $\Sigma$ as in Definition~\ref{protoMom:def}. The
triple $(M,\Sigma,\Delta)$ is called a \emph{weak
protoMom-structure} if $\Delta$ has no $0$- or $3$-handles and all
its $2$-handles have valence $3$ (whereas no restriction is placed
on the valences of $1$-handles).
\end{defn}

\begin{rem}
\emph{Our definition of weakness differs from that given for
Mom-structures in \cite{Mom1}, where a weak Mom-structure is one
where $2$-handles are allowed to have valence $2$ (for us it is
always $3$), while $1$-handles are still required to have valence at
least $2$ (whereas we drop this assumption in the definition of
weakness).}
\end{rem}

\paragraph{Lateral tori and internal protoMom-structures}
Given a (weak) proto-\-Mom-structure $(M,\Sigma,\Delta)$, we call
any toric boundary component of $M$ that is not contained in
$\Sigma$ a \emph{lateral torus} of the given protoMom-structure.
Clearly, if $\Sigma$ does not have toric boundary components, then
any torus in $\partial M$ is lateral.

\begin{defn}
Let $N$ be a $3$-manifold with non-empty boundary, and let
$(M,\Sigma,\Delta)$ be a protoMom structure on a submanifold
$M\subset N$. Then $(M,\Sigma,\Delta)$ is called an \emph{internal
protoMom structure on $N$} if $N\setminus M$ consists of a collar of
$\partial N$ and possibly some solid tori.
\end{defn}
\noindent The meaning of this definition is simply that $N$ can be
obtained from $M$ by Dehn-filling some of the lateral tori of $M$.

Note that in what follows our main focus will mostly be on full weak
protoMom-structures internal on a certain fixed $3$-manifold; we
conclude this section by describing three other notions that will
provide us with the tools necessary to obtain our results, namely
\emph{ideal triangulations}, \emph{special spines} and their
\emph{singular graphs}, and the natural \emph{duality} between the
former two classes of objects.

\paragraph{Ideal triangulations} By an \emph{ideal triangulation}
of a compact $3$-manifold $N$ with boundary we mean a realization of
$N$ as the result of first gluing a finite number of tetrahedra
along a complete system of simplicial pairings of their lateral
faces, and then removing small open neighbourhoods of all the
vertices (such neighbourhoods must be small enough so that their
closures be disjoint). Note that $N$ is thus decomposed into
\emph{truncated tetrahedra} (see Fig.~\ref{duality:fig}-left), not
into actual ones. Observe also that we allow multiple and
self-adjacencies of the tetrahedra.

\paragraph{Special spines} A compact $2$-dimensional polyhedron $P$ is called \emph{special}
if the following two conditions hold. First, the link of each point
should be homeomorphic to one of the following 1-dimensional
polyhedra:
\begin{enumerate}
\item[(a)] a circle;
\item[(b)] a circle with a diameter;
\item[(c)] a circle with three radii.
\end{enumerate}
(See Fig.~\ref{neighborhds_pts_simple:fig} where the corresponding
three possible types of regular neighbourhoods of a point of $P$ are
shown).
    \begin{figure}
    \begin{center}
    \includegraphics[scale=0.5]{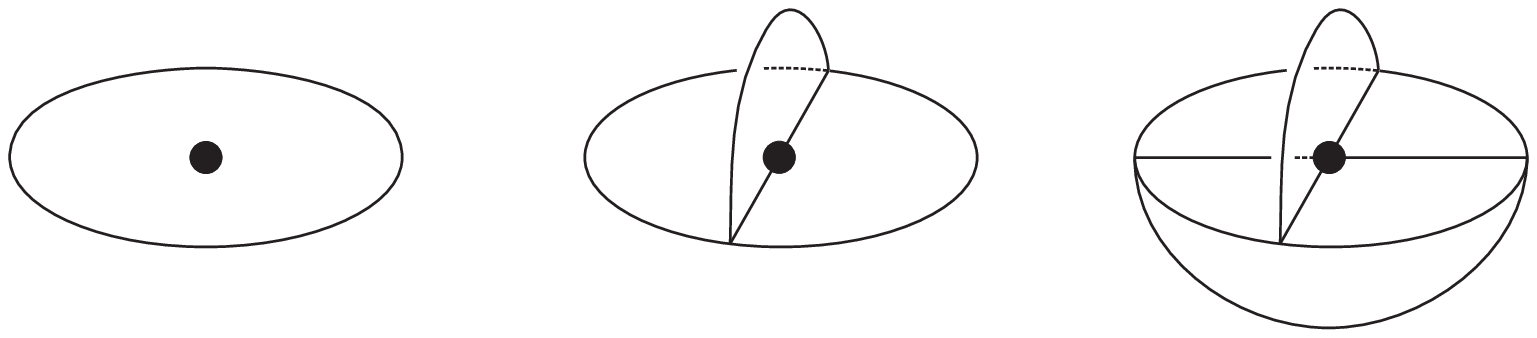}
    \mycap{Neighbourhoods of points in a special polyhedron.}
    \label{neighborhds_pts_simple:fig}
    \end{center}
    \end{figure}
Second, the components of the set of points with link of type (a)
are open discs, while the components of the set of points with link
of type (b) are open segments.

The components just described are called \emph{faces} and
\emph{edges}, respectively, and the points with link of type (c) are
called \emph{vertices}. The points with link of type (b) or (c) are
called \emph{singular}, and the set of all singular points of $P$ is
denoted by $S(P)$, which is a four-valent graph (without circular
components) that we will refer to as the \emph{singular graph} of
$P$.

Let now $N$ be a $3$-manifold with non-empty boundary, and let $P$
be a special polyhedron embedded in $\Int N$. We say that $P$ is a
\emph{special spine} of $N$ if $N\setminus P$ is an open collar of
$\partial N$.

\paragraph{Duality} The final notion that we need is the well-known duality between
ideal triangulations of $3$-manifolds with boundary and their
special spines, summarized in the next statement.

\begin{prop}
Let $N$ be a compact $3$-manifold with non-empty boundary. Then the
set of ideal triangulations of $N$ corresponds bijectively to the
set of special spines of $N$ via the correspondence shown in
Fig.~\ref{duality:fig}.
\end{prop}

    \begin{figure}
    \begin{center}
    \includegraphics[scale=0.5]{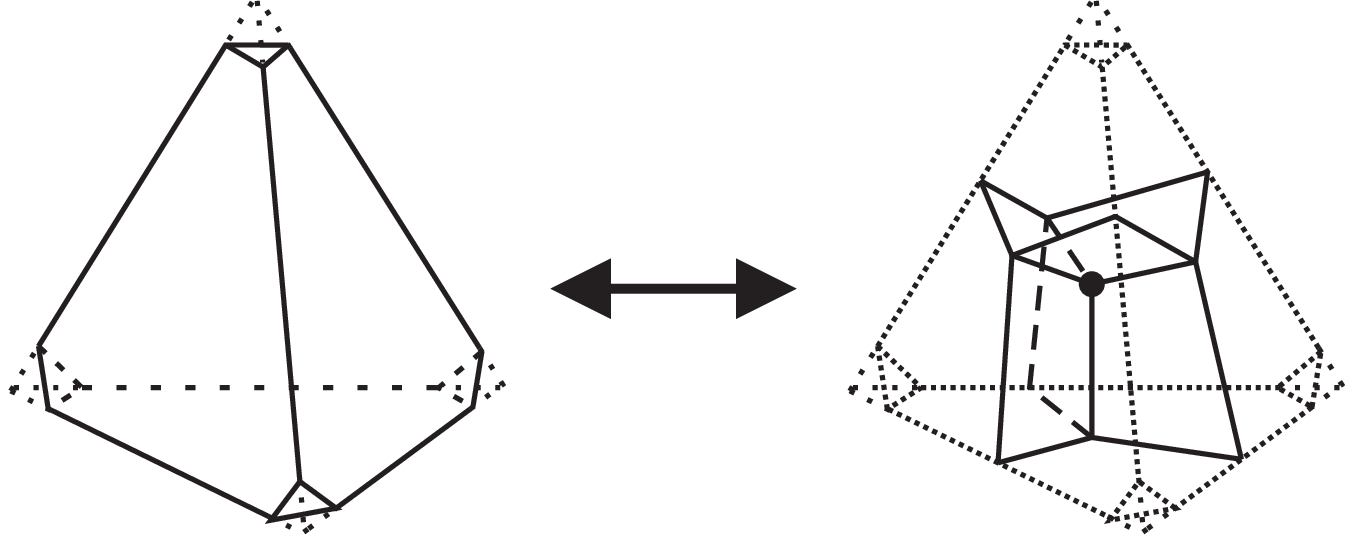}
    \mycap{Duality between ideal triangulations and special polyhedra.}
    \label{duality:fig}
    \end{center}
    \end{figure}

Given an ideal triangulation $\tau$ of $N$, we will denote the
special spine dual to $\tau$ by $P_{\tau}$.

\section{Ideal triangulations and\\ internal
protoMom-structures}\label{trn:e:mom:sec}

In this section we describe a procedure that, given an ideal
triangulation of a manifold $N$ with non-empty boundary (without
spherical components), allows us to obtain an internal weak
protoMom-structure on $N$. We then give a slightly more general
notion of a so-called \emph{triangulation-induced}
protoMom-structure (that we usually call $\tau$-induced, where
$\tau$ stands for a specific triangulation), and we show that all
full weak protoMom-structures internal on a given $N$ are actually
induced by some ideal triangulation of $N$. The proof of this fact
is constructive, and will be used in Section~\ref{moves:main:sec} to
study combinatorial moves relating different protoMom-structures
internal on the same fixed manifold $N$.

\subsection{Triangulation-induced protoMom-structures}\label{trn:2:protomom:subsec}

Fix a compact connected orientable $3$-manifold $N$ such that
$\partial N=\Sigma$ is a non-empty surface without spherical
components. Let $\tau$ be an arbitrary ideal triangulation of $N$,
seen as a decomposition of $N$ into truncated tetrahedra. Then
$\tau$ gives rise to a number of internal protoMom-structures on
$N$, which are constructed in the following way.

\paragraph{From an ideal triangulation to a protoMom-structure: constructive approach}
Consider the $3$-manifold $M'$ obtained by removing from each
tetrahedron of $\tau$ a small open ball around the centre of the
tetrahedron. The triangulation $\tau$ naturally induces on $M'$ a
general-based handle decomposition $\Delta'$ obtained by taking the
collar of $\Sigma$ and thickening each edge of $\tau$ to a
$1$-handle and each face to a $2$-handle.

We can construct a (non-unique) weak protoMom-structure
$(M,\Sigma,\Delta)$ with $M\subset M'$ and $\Delta$ induced by
$\Delta'$ by starting from $\Delta'$ and repeatedly deleting
$2$-handles $\alpha$ satisfying one of the following conditions:
\begin{enumerate}
    \item[(a)] the ends of the cocore of $\alpha$ lie on two distinct spherical
    components of the boundary, or
    \item[(b)] one end of the cocore of $\alpha$ lies on a spherical component and the
    other one on a toric component, or
    \item[(c)] both ends of the cocore of $\alpha$ lie on the same spherical
    component.
  \end{enumerate}
We do this until a handle decomposition of a manifold bounded by
$\Sigma$ and some tori is reached.

Denote the manifold obtained as a result by $M$ and the subset of
the handles of $\Delta'$ that are contained in $M$ by $\Delta$. Then
it is evident that $(M,\Sigma,\Delta)$ thus obtained is a weak
protoMom-structure internal on $N$.

\begin{rem}
\emph{Observe that the procedure described above involves the
minimal possible number of handle deletions required to obtain an
internal protoMom-structure composed of handles belonging to
$\Delta'$ (this fact will be given a precise formulation in
Section~\ref{moves:main:sec}, see Definition~\ref{trn:max:str:def}
and Lemma~\ref{max:e:mom:lem}). Note however that canceling a
$1$-handle of valence $1$ with the $2$-handle incident to it
transforms a weak protoMom-structure into a weak protoMom-structure,
so performing such cancelations intermittently with the above
operation allows us to obtain a much wider class of
protoMom-structures from a given triangulation. An \emph{a priori}
still wider class of protoMom-structures could be obtained by
allowing the removal of $1$-handles of valence $0$ intermittently
with the removal of $2$-handles, but we will not consider this
latter procedure in the paper.}
\end{rem}

\begin{rem}\label{obtain:genuine:rem}
\emph{It will be an easy consequence of Lemma~\ref{proto:2:max:lem}
below that all ``genuine'' (\emph{i.e.}, non-weak) full
protoMom-structures can be obtained from some triangulation by first
removing $2$-handles via the operation described and then
successively canceling $1$-handles of valence $1$ until none such is
left.}
\end{rem}

We now give a formal definition of a notion already referred to
above.

\begin{defn}\label{trn:induced:def}
Let $N$, with $\partial N=\Sigma$, $\tau$, and $\Delta'$ be as
above. A weak protoMom-structure $(M,\Sigma,\Delta)$ internal on $N$
is said to be \emph{$\tau$-induced} if the set of handles of
$\Delta$ is a subset of that of $\Delta'$.
\end{defn}

Note that, while any protoMom-structure obtained as described above
is of course $\tau$-induced, this notion is \emph{a priori} more
general. Indeed, suppose that we have a $\tau$-induced
protoMom-structure $(M,\Sigma,\Delta)$ containing a subset of $1$-
and $2$-handles such that: 1) the subset forms an annular sheet in
the sense of~\cite{Mom1}; 2) the lateral boundary of this sheet
consists of two annuli belonging to two different lateral boundary
components of the ambient manifold $M$; 3) each such lateral annulus
is non-trivial in the respective lateral torus. Then removing such
sheet still yields a $\tau$-induced protoMom-structure, which may or
may not be obtainable via a sequence of the operations described
above.

\paragraph{Geometric protoMom-structures and canonical decompositions}
We briefly observe that a certain adjustment of the above
constructive procedure allows to obtain a number of weak
protoMom-structures internal on non-closed hyperbolic $3$-manifolds
without annular cusps and geometric in the sense of~\cite{Mom2,Mom4}
(\emph{i.e.}, such that the cores of all the $1$- and the
$2$-handles are respectively geodesic arcs and geodesic hexagons).
The adjustment is as follows. We consider, as appropriate, either
the Epstein-Penner~\cite{EpPe} or the Kojima
decomposition~\cite{koji1,koji2} of $N$, and we subdivide each face
of each polyhedron into triangles by geodesic arcs. Then we
construct $\Delta'$ by taking thickenings of the edges of the
polyhedra and of the added arcs as $1$-handles, and the thickenings
of the triangles as $2$-handles ($M'$ is then the ambient space of
$\Delta'$), and we apply precisely the same procedure as above (with
or without the collapses) to obtain a geometric weak
protoMom-structure from~$\Delta'$.

\subsection{From an internal protoMom-structure\\ to an ideal
triangulation}\label{mom:2:trn:subsec}

In this section we will establish the following result.

\begin{thm}\label{int:protoMom:2:trn:thm}
Let $(M,\Sigma,\Delta)$ be a full weak protoMom-structure internal
on a compact connected orientable $3$-manifold $N$ with $\partial
N=\Sigma$. Then there exists an ideal triangulation $\tau$ of $N$
such that $(M,\Sigma,\Delta)$ is $\tau$-induced.
\end{thm}

To prove this theorem, we will need several auxiliary tools,
including one well-known construction described in the next
paragraph.

\paragraph{Layered triangulations} Given a 3-manifold $M'$ with a
toric boundary component $T$ triangulated into two triangles, a
\emph{layering along an edge $e$} (see \cite{layered} and the
references therein) of this triangulation consists in gluing a
standard tetrahedron $\Delta_3$ to $T$ via a simplicial
homeomorphism between $\alpha\cup\beta$, where $\alpha$ and $\beta$
are the two triangular faces in the triangulation of $T$ (and are
therefore adjacent to $e$), and two faces $\alpha'$ and $\beta'$ in
the boundary of $\Delta_3$, such that the common edge
$\alpha'\cap\beta'$ is glued to $e$. Clearly, this operation changes
the triangulation of $T\subset\partial M'$ but not the manifold $M'$
itself; we will need it to describe the so-called \emph{layered
triangulations} of solid tori.

A layered triangulation of a solid torus $H$ is any triangulation of
$H$ obtained by taking the one-tetrahedron triangulation of $H$
shown in Fig.~\ref{trn:solid:tor:easy:fig}-left
    \begin{figure}
    \begin{center}
    \includegraphics[scale=0.5]{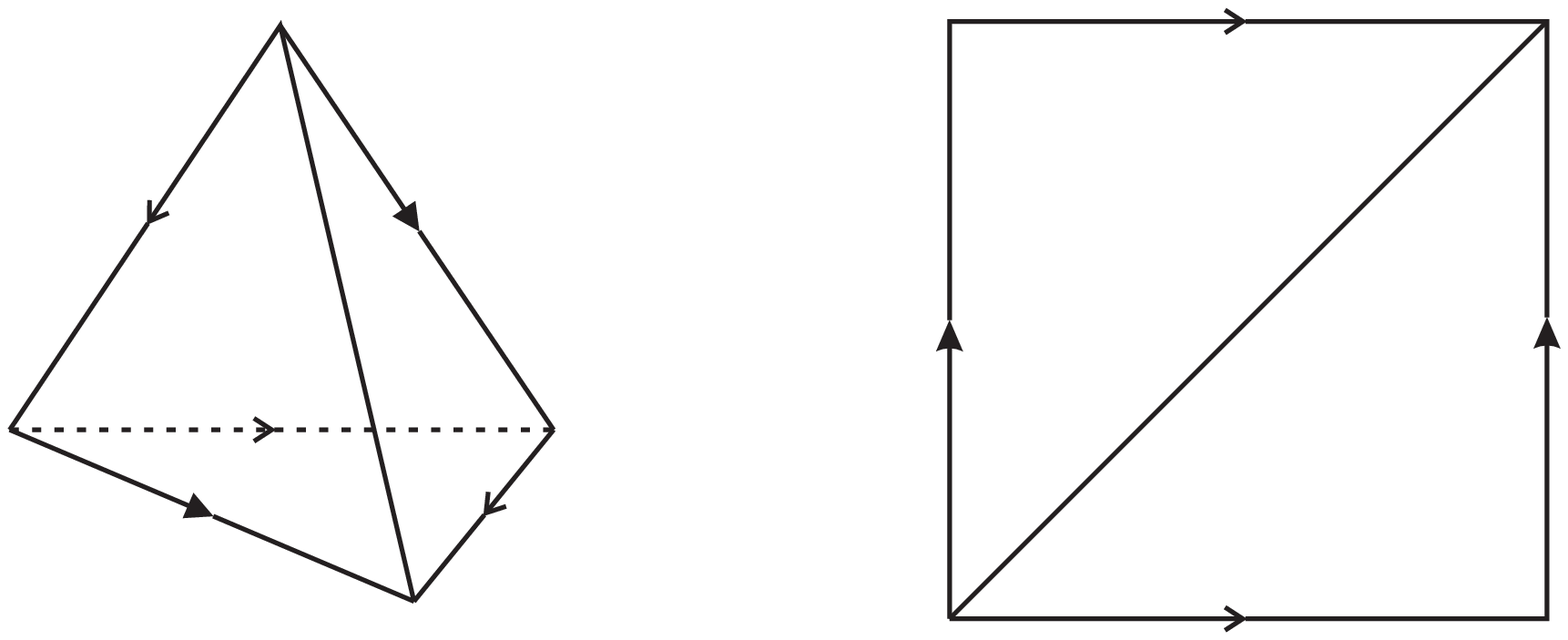}
    \mycap{Left: one-tetrahedron triangulation of a solid torus. The two back faces
    are glued together according to the labels of the edges. Right: the resulting
    triangulation of the boundary torus.}
    \label{trn:solid:tor:easy:fig}
    \end{center}
    \end{figure}
and then performing several successive operations of layering along
edges in $\partial H$. If $\tau$ is a layered triangulation of $H$
then its restriction to the boundary of $H$ is always of the form
shown in Fig.~\ref{trn:solid:tor:easy:fig}-right; the operation of
layering changes the position of this triangulation with respect to
the meridinal disc of $H$. More precisely, the set of triangulations
of $\partial H$ as in Fig.~\ref{trn:solid:tor:easy:fig}-right is in
bijective correspondence with the set of isotopy classes of
embeddings of the $\theta$-curve in $\partial H$ having a disc as
the complement, where the bijection is by taking the dual graph to
the $1$-skeleton of the triangulation, as in
Fig.~\ref{trn2:dual:theta:fig}.
    \begin{figure}
    \begin{center}
    \includegraphics[scale=0.3]{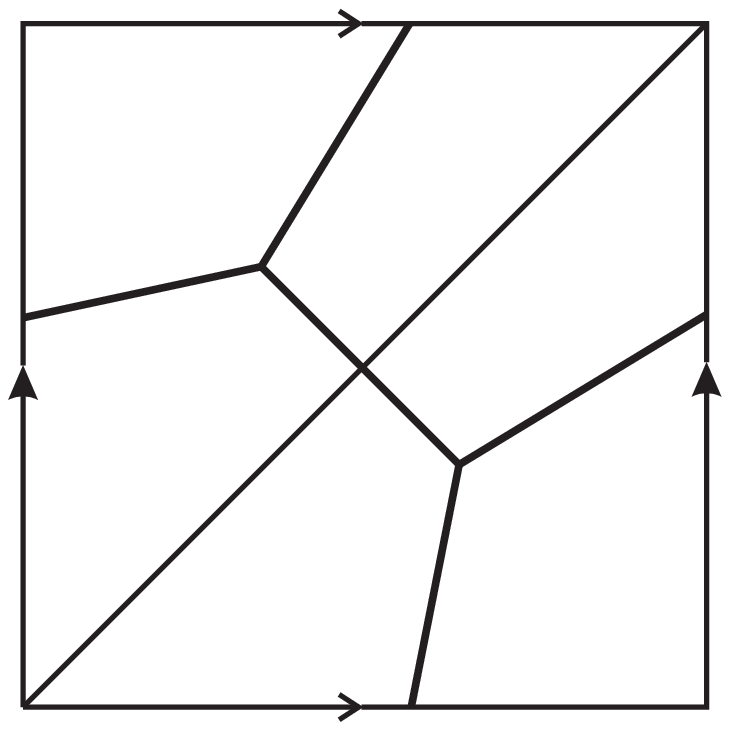}
    \mycap{The dual $\theta$-curve.}
    \label{trn2:dual:theta:fig}
    \end{center}
    \end{figure}

To proceed, we will need the following known result.

\begin{lemma}\label{exist:layered:lem}\emph{(\cite{layered})}
For every $\theta$-curve in the boundary of a solid torus $H$ having
a disc as the complement there exists a layered triangulation of $H$
such that the $1$-skeleton of its restriction to $\partial H$ is
dual to the $\theta$-curve.
\end{lemma}

\paragraph{Induced triangulations of lateral tori} Observe that a full
weak protoMom-structure cannot contain $1$-handles of valence $0$.
Hence, each lateral torus $T_i$ of the ambient manifold $M$ of such
a protoMom-structure admits a natural decomposition into disks of
the following 3 types:
\begin{itemize}
  \item lakes (contained in $\Sigma\times\{1\}$);
  \item strips of the form $\ell\times[0,1]$ with $\ell\subset\partial
  D^2$ for some $1$-handle $D^2\times[0,1]$;
  \item hexagons that are lateral sides of $2$-handles.
\end{itemize}
This decomposition induces a natural triangulation of $T_i$ obtained
by compressing each lake to a point and each strip $\ell\times[0,1]$
to an arc $\{*\}\times[0,1]$. We denote this triangulation by
$\delta'(T_i)$.

\paragraph{Moves on triangulations of surfaces} We now introduce several
moves on triangulations of surfaces, that will be used to prove
Theorem~\ref{int:protoMom:2:trn:thm}. Let $T$ be a closed surface
endowed with a triangulation $\delta$.
\begin{enumerate}
  \item[(s1)] Let $v$ be a vertex of $\delta$, and let $e'$, $e''$ be
  two distinct edges incident to $v$. Cut $T$ open along $e_1\cup v\cup e_2$
  and fill the resulting square by two triangles sharing an edge
  in the original position of $e_1\cup v\cup e_2$. We say that the
  move (s1) \emph{is performed at $v$ along $e',e''$.}

  \item[(s2)] Let $e$ be an edge of $\delta$ belonging to two distinct
  triangles. Remove $e$ and replace it by the other diagonal $e'$ of the
  resulting square (``flipping'' $e$). We say that the move (s2)
  \emph{is performed at $e$}.

  \item[(s3)] Let $v$ be a vertex of $\delta$ incident to exactly
  $3$ distinct triangles. Merge the three triangles into a single
  one, removing $v$ and the edges incident to it.
\end{enumerate}
The moves (s1-s3) are shown in Fig.~\ref{moves:triang:fig}.
    \begin{figure}
    \begin{center}
    \includegraphics[scale=0.5]{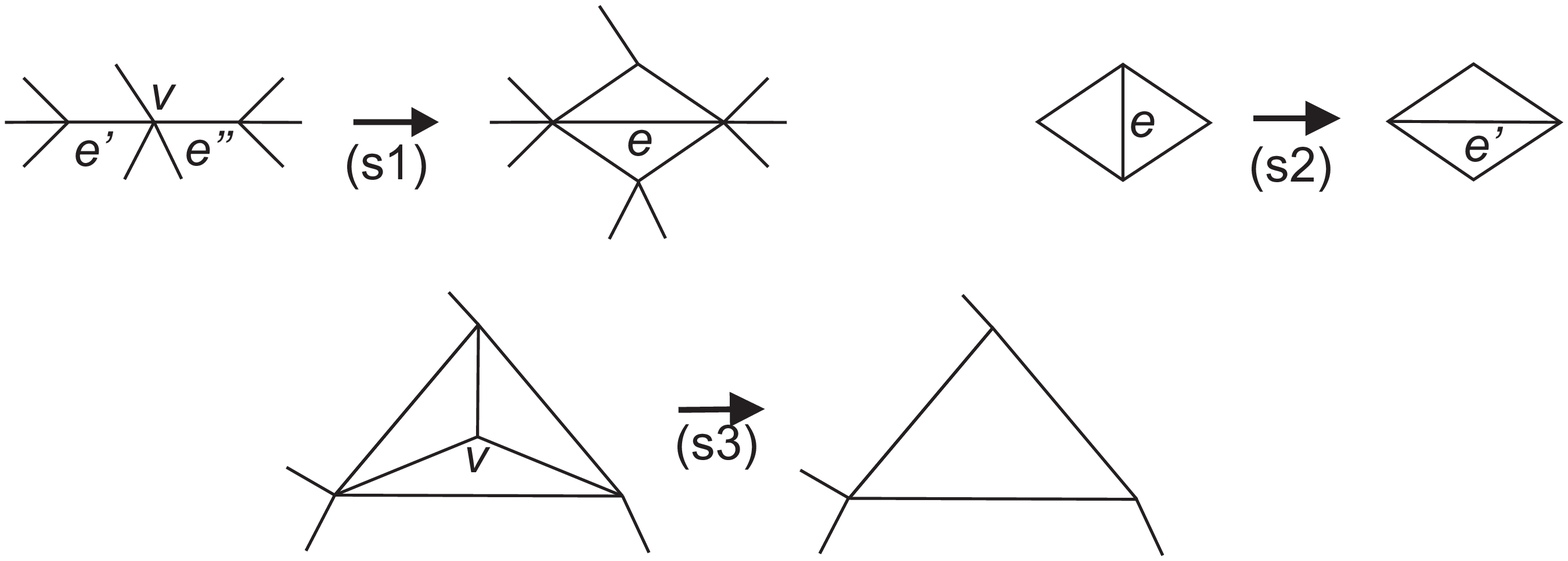}
    \mycap{The moves (s1), (s2), (s3).}
    \label{moves:triang:fig}
    \end{center}
    \end{figure}

We will most often apply the move (s1) at a vertex $v'$ along two
edges $e',e''$ such that $e'$ is incident to a vertex $v$ of valence
$1$, see Fig.~\ref{move:s1:sp:case:fig}.
    \begin{figure}
    \begin{center}
    \includegraphics[scale=0.5]{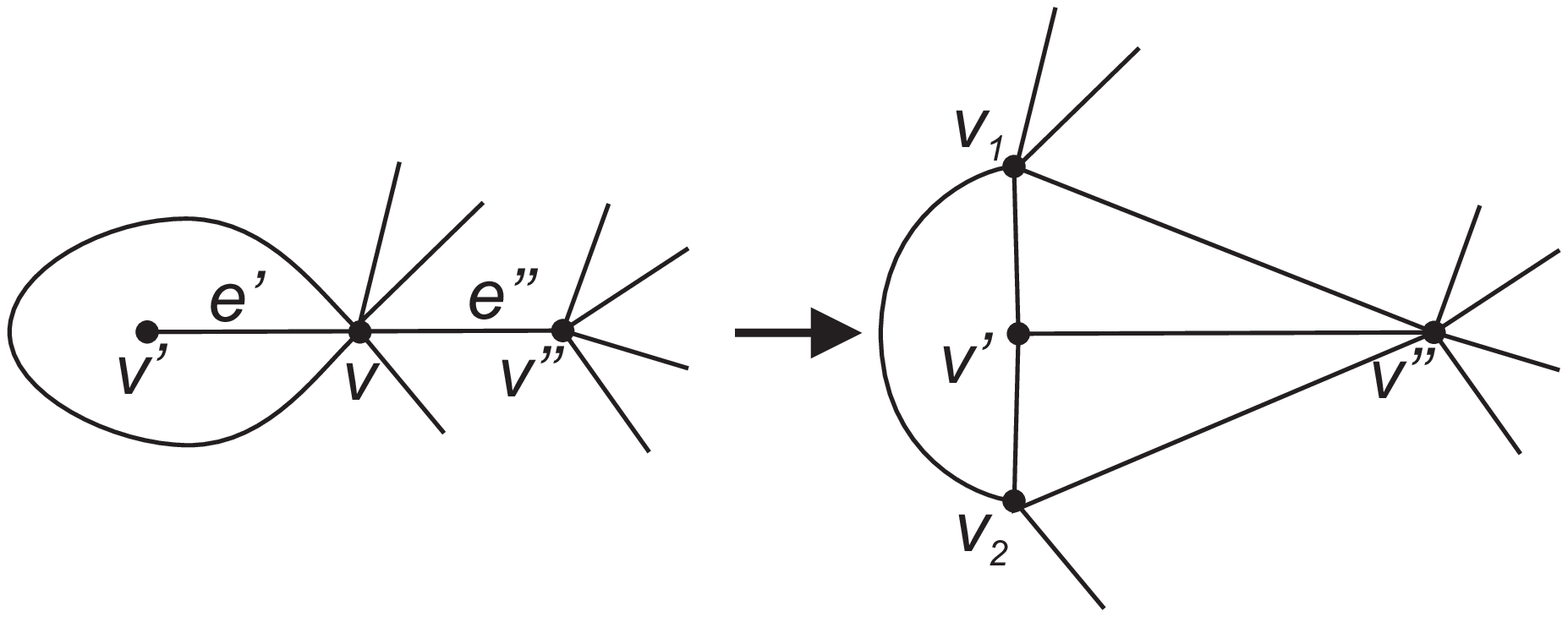}
    \mycap{A special case of move (s1).}
    \label{move:s1:sp:case:fig}
    \end{center}
    \end{figure}
Observe that the move destroys $v$ without creating vertices of
valence $1$ or $2$.

Below we will often use the following composition of moves.

\begin{enumerate}
  \item[(s1$'$)] Let $v$ be a vertex of $\delta$ of valence $2$. Denote the two
  edges incident to $v$ by $e'$ and $e''$. The move (s1$'$) consists in
  performing the move (s1) at $v$ along $e',e''$ and subsequently
  performing the move (s2) along the newly created edge~$e$.
\end{enumerate}
The effect of this move on the triangulation is shown in
Fig.~\ref{move:val2:fig}.
    \begin{figure}
    \begin{center}
    \includegraphics[scale=0.4]{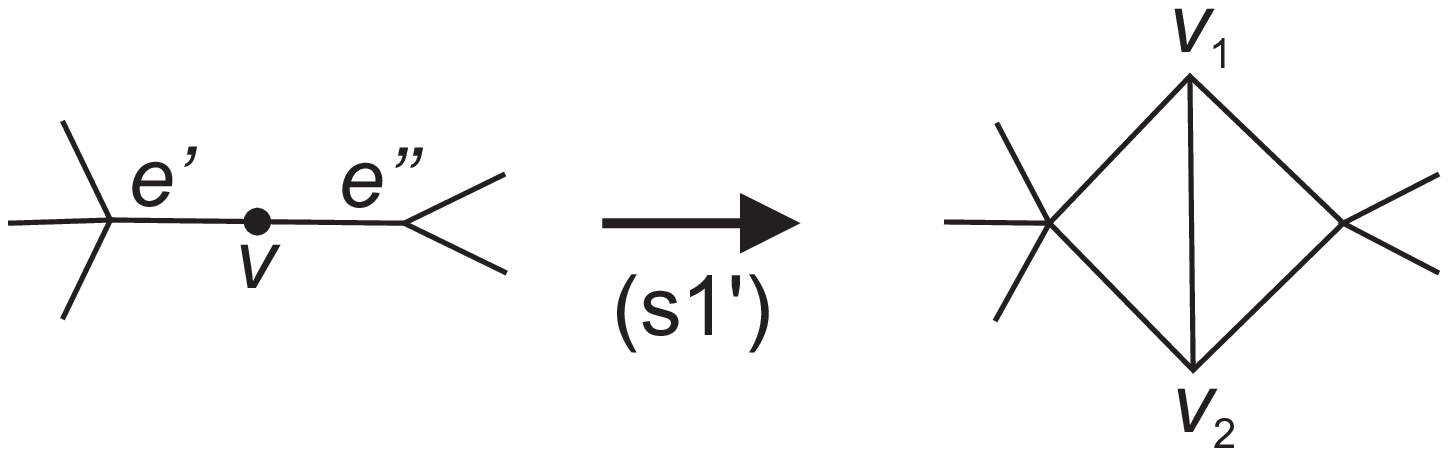}
    \mycap{The move (s1').}
    \label{move:val2:fig}
    \end{center}
    \end{figure}
Observe that the move decreases the number of vertices of valence
$2$ by one, and does not create any vertices of valence $1$.

\paragraph{Simplifying a triangulation} The key ingredient in the
proof of Theorem~\ref{int:protoMom:2:trn:thm} will be the following
fact.

\begin{lemma}\label{trn:tor:2:tor:prelim:lem}
Let $\delta$ be a triangulation of a two-dimensional torus. Then
there exists a sequence of moves (s1), (s2), (s3) transforming
$\delta$ into a triangulation with exactly two triangles.
\end{lemma}

The desired triangulation is the one shown in
Fig.~\ref{trn:solid:tor:easy:fig}-right.

\begin{proof}
The strategy of the proof is to use first the moves (s*) to get rid
of all vertices of valence $1$ and $2$; note that, since we will use
the move (s1$'$), this may increase the total number of vertices. We
then show that a triangulation where all vertices have valence at
least $3$, can be transformed into the desired one by moves (s2) and
(s3), and the argument is based on the fact that the move (s3)
strictly decreases the number of vertices.\vspace{4mm}

\noindent (1) Assume first that $\delta$ contains a vertex $v$ of
valence $1$. Then it is easy to see that we have a situation as in
Fig.~\ref{move:val1:torus:fig}-left, possibly with $v_1=v_2$.
    \begin{figure}
    \begin{center}
    \includegraphics[scale=0.3]{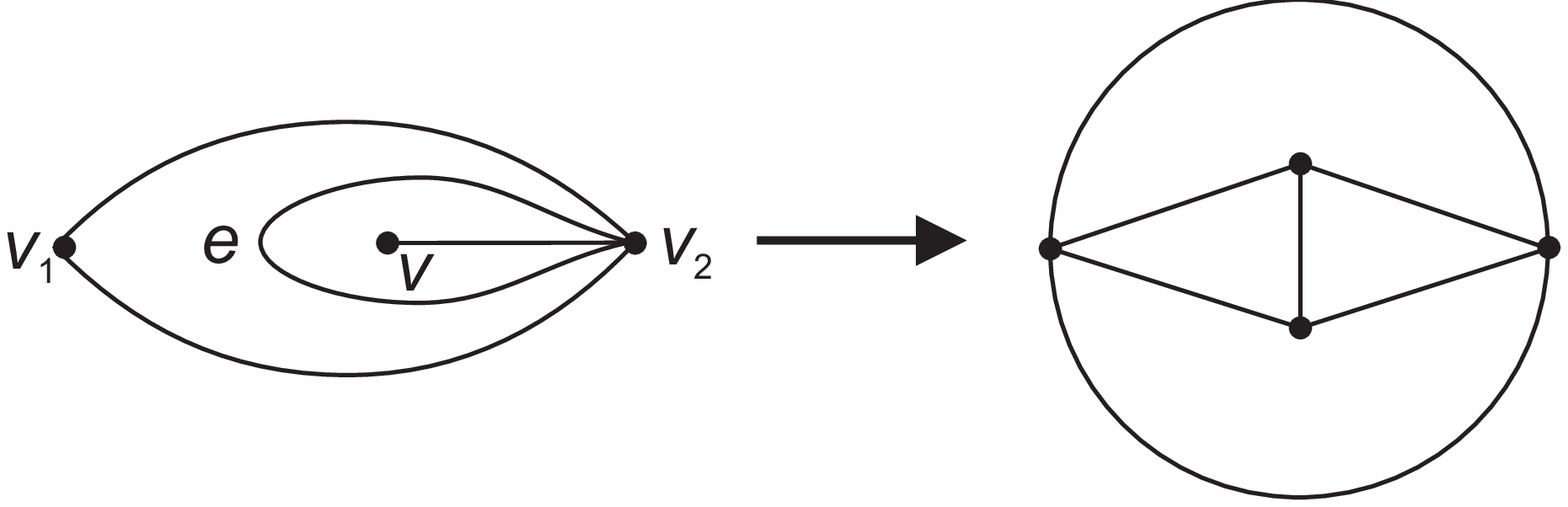}
    \mycap{Destroying a vertex of valence $1$; note that $v_1$ and $v_2$ may coincide.}
    \label{move:val1:torus:fig}
    \end{center}
    \end{figure}
We then perform the move (s2) at the edge $e$, followed by the move
(s1$'$) at the vertex $v$, as shown in
Fig.~\ref{move:val1:torus:fig}. This destroys the vertex $v$ of
valence $1$ without creating new vertices of valence $1$ or $2$, so
we can repeat the procedure until no vertices of valence $1$ are
left.\vspace{4mm}

\noindent (2) Assume now that $\delta$ has vertex $v$ of valence $2$
but not of valence $1$. We then perform (s1$'$) at $v$ and proceed
until no vertex of valence $1$ or $2$ remains.\vspace{4mm}

\noindent (3) Assume now that all vertices of $\delta$ have valence
at least $3$. The proof is by induction on the number $k$ of
vertices of $\delta$. The base of induction is $k=1$ and it follows
from the Euler characteristic argument that in this case we indeed
have the triangulation with two triangles. Assume that
$k>1$.\vspace{4mm}

\noindent\emph{Case 3.1.} Suppose that $\delta$ contains a vertex
$v$ of valence precisely $3$. It is easy to see that in this case
the closure of any edge incident to $v$ is a segment. There are now
the following two possibilities:\vspace{3mm}

\noindent\emph{Case 3.1.a.} There exists a vertex $v$ of valence $3$
such that the three vertices $v_1$, $v_2$, $v_3$ joined to $v$ by an
edge are all of valence at least $4$. Observe first that if two of
these vertices, say $v_2$ and $v_3$, coincide then the valence of
$w=v_2=v_3$ is at least 5. We now perform the move (s3) at $v$, the
effect of doing which is that $v$ disappears and the valences of
$v_1$, $v_2$, $v_3$ are decreased by $1$, if they are all distinct;
if they are not, the valence of $w$ is decreased by $2$. In either
case, by the assumption these new valences are still at least $3$.
In addition, observe that the valences of all the other vertices are
not affected, and that the total number of vertices in $\delta$ is
decreased by~$1$.\vspace{3mm}

\noindent\emph{Case 3.1.b.} The second possibility is that for any
vertex $v$ of valence $3$ at least one of the vertices $v_1$, $v_2$,
$v_3$ joined to $v$ by an edge has valence $3$, with $v_i$'s not
necessarily all distinct. Choose any such $v$ and observe that we
must have the situation as shown in Fig.~\ref{val3:sp:case:fig},
    \begin{figure}
    \begin{center}
    \includegraphics[scale=0.3]{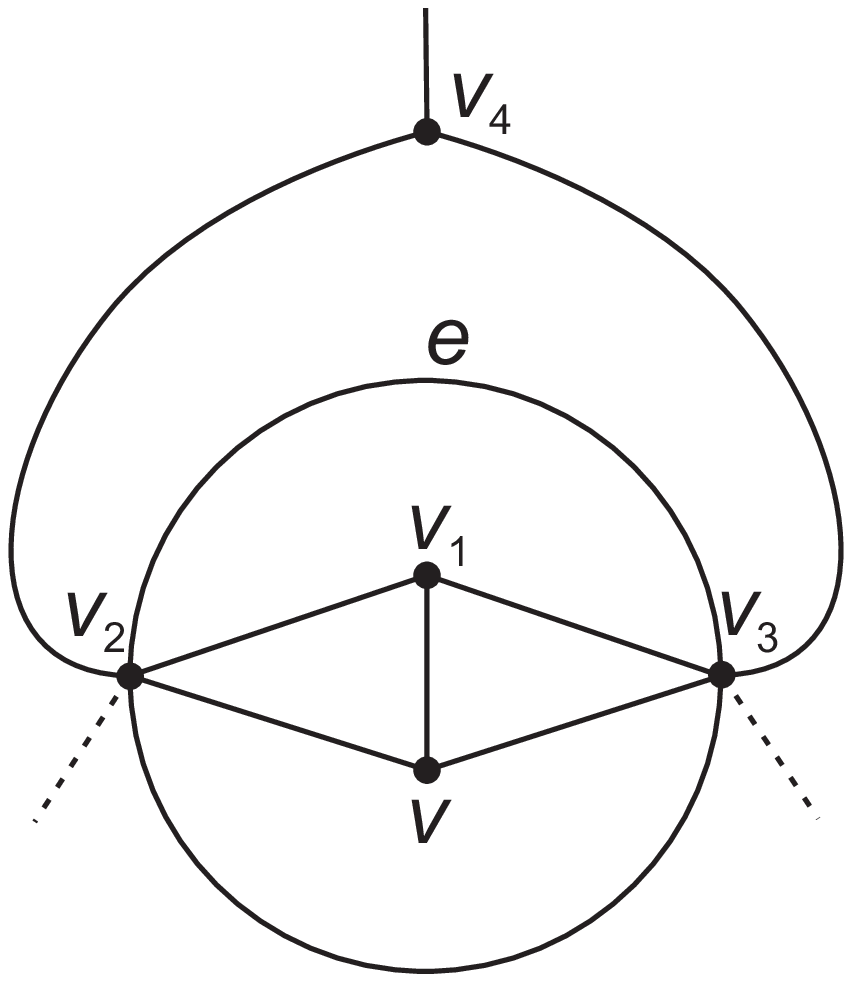}
    \mycap{Two vertices of valence $3$ joined by an edge.}
    \label{val3:sp:case:fig}
    \end{center}
    \end{figure}
with the valences of $v_2$, $v_3$ being at least $5$ and where we
may or may not have $v_2=v_3$. Denote by $e$ one of the edges
joining $v_2$ to $v_3$; then $e$ belongs to two triangles, one with
vertices $v_1$, $v_2$, $v_3$ and the other with vertices $v_2$,
$v_3$, $v_4$; we also have that $v_4\neq v_1$. Thus, performing the
move (s2) at $e$ has the following effect on the set of valences:
\begin{itemize}
  \item the valence of $v_1$ becomes $4$;
  \item the valence of $v_4$ is increased by $1$;
  \item if $v_2\neq v_3$ then the valence of each of them decreases by
  $1$, and if they coincide then the total valence is decreased by $2$,
  however, in this case it was at least $6$ to begin with.
\end{itemize}
Notice that no other vertices are affected by the move and that
after the move the valences of $v_2$ and $v_3$ (or that of
$v_2=v_3$) are still at least $4$. Therefore we can proceed as in
Case 3.1.a.\vspace{4mm}

\noindent\emph{Case 3.2.} Suppose that all vertices of $\delta$ have
valence at least $4$. Again, there are two
possibilities:\vspace{3mm}

\noindent\emph{Case 3.2.a.} There is a vertex such that the closure
of any edge incident to it is a segment. Let $v$ be a vertex of
minimal valence among all such vertices. Denote the edges incident
to $v$ by $e_1$, $\ldots$, $e_k$ in the clockwise order around $v$
and let $v_i$ be the other endpoint of $e_i$ for $i=1,\ldots,k$ (we
note that some of the $v_i$'s may coincide). Perform now the move
(s2) consecutively along the edges $e_2$, $e_3$, $\ldots$,
$e_{k-2}$, followed by the move (s3) at the vertex $v$, see
Fig.~\ref{case:3:2:a:fig}.
    \begin{figure}
    \begin{center}
    \includegraphics[scale=0.4]{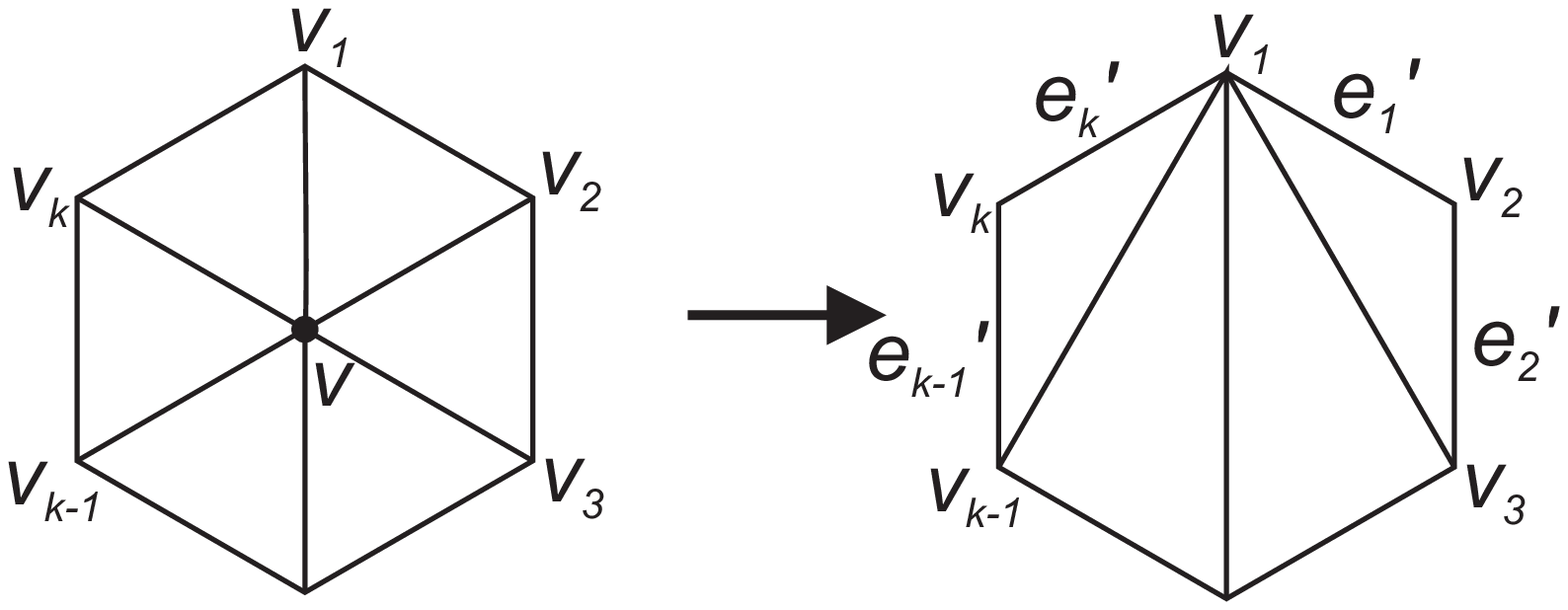}
    \mycap{Destroying vertex $v$.}
    \label{case:3:2:a:fig}
    \end{center}
    \end{figure}
Observe that:
\begin{itemize}
  \item the valence of any vertex (among $v_i$'s) distinct from $v_2$ and $v_k$
  does not decrease;
  \item if $v_2\neq v_k$ then the valence of each of these vertices is
  decreased at most by $1$, hence it is still at least $3$.
\end{itemize}

Suppose now that $v_2=v_k=w$, and that the valence of $w$ is at most
$2$; this implies in particular that $w\neq v_1$ and that the
valence of $w$ before the move was at most $4$. We claim that in
this case $w$ cannot be incident to a loop; indeed, if it were, the
loop edge would be distinct from the edge $e_1'$, hence after the
moves the valence of $w$ would be at least $3$. Now, since $v$ was a
vertex of minimal valence among those edges not incident to any
loop, this implies that $k=4$ and that among the edges $e_1'$,
$e_2'$, $e_3'$, $e_4'$ there are precisely two distinct ones.
However, this would imply that we already have a triangulation of
the torus into two triangles.

Since the total number of vertices diminishes, we can repeat the
entire procedure, starting from (3), until we either reach a
triangulation with exactly two triangles or encounter the situation
as in Case 3.2.b below.\vspace{3mm}

\noindent\emph{Case 3.2.b} Suppose that all vertices are incident to
at least one edge whose closure is a loop. In this case all such
loops are nontrivial in $T$ because if one of them were the boundary
of a disc, an innermost loop in this disc would contradict the fact
that we are dealing with a triangulation. Then one can see directly
that, up to applying the move (s2), we have a situation as shown in
Fig.~\ref{trn:all:loops:fig}
    \begin{figure}
    \begin{center}
    \includegraphics[scale=0.4]{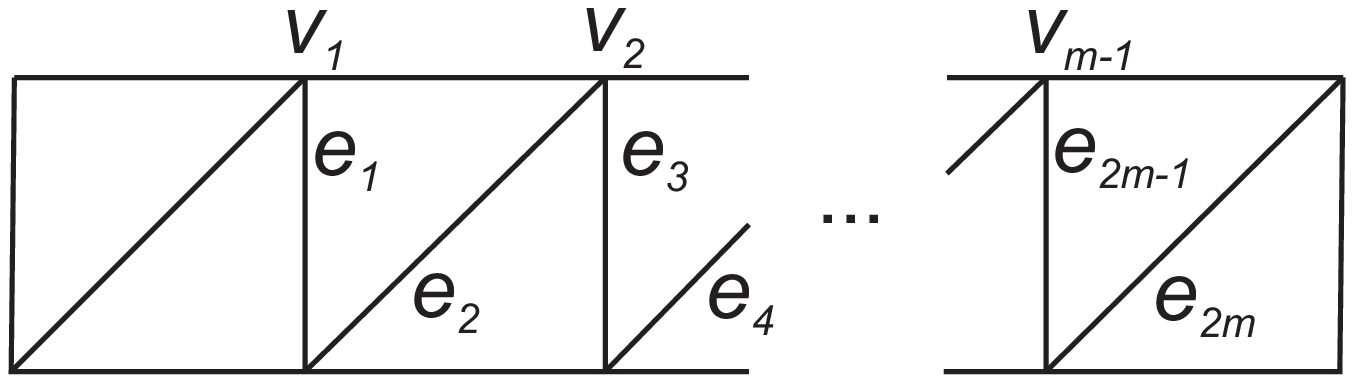}
    \mycap{A series of particular triangulations of a torus; the opposite sides of
    the rectangle are identified forming the torus.}
    \label{trn:all:loops:fig}
    \end{center}
    \end{figure}
(more precisely, the triangulation either will be exactly as in the
figure or can be obtained from it by performing (s2) at the
diagonals of some of the squares). In this case the proof is by
induction on the number $m$ of parallel loops.\vspace{3mm}

\noindent If $m=1$, we already have the desired triangulation. If
$m>1$ then we perform the move (s2) at $e_1$, then (s2) at $e_2$,
and finally (s3) at $v$. This again yields a triangulation as in
Case 3.2.b but the number $m$ decreases by~$1$, whence the
conclusion.
\end{proof}

\paragraph{Constructing triangulations of solid tori} Let now $H$ be
a solid torus, and let $\delta$ be a triangulation of $T=\partial
H$. We define the following moves consisting in inserting
\emph{inside} $H$ some $2$- or $3$-dimensional simplices.
\begin{enumerate}
  \item[(\^{s}1)] Let $v$ be a vertex of $\delta$, and let $e'$, $e''$ be
  two distinct edges adjacent to $v$. The move (\^{s}1) consists in
  gluing to $T$, inside $H$, a triangle via a
  piecewise-linear identification of two of its edges with $e'\cup
  v\cup e''$. We say that the move (\^{s}1) \emph{is performed at $v$ along $e',e''$.}

  \item[(\^{s}2)] Let $e$ be an edge of $\delta$ adjacent to two distinct faces
  $\alpha$ and $\beta$. The move (\^{s}2) consists in
  gluing to $T$, inside $H$, a tetrahedron via a piecewise-linear
  identification of two of its faces with $\alpha\cup\beta$. We say
  that the move (\^{s}2) \emph{is performed at $e$}; note that this is the same
  operation as the one performed in the construction of any layered
  triangulation.

  \item[(\^{s}3)] Let $v$ be a vertex of $\delta$ incident to exactly
  $3$ distinct edges. Denote the three triangles incident to $v$ by
  $\alpha_1,\alpha_2,\alpha_3$. The move (\^{s}3) consists in
  gluing to $T$, inside $H$, a tetrahedron via a piecewise-linear
  identification of three of its faces with
  $\alpha_1\cup\alpha_2\cup\alpha_3$.
\end{enumerate}
We also define the following composite move.
\begin{enumerate}
  \item[(\^{s}1$'$)] Let $v$ be a vertex of $\delta$ of valence $2$. Denote the two
  edges adjacent to $v$ by $e'$ and $e''$. The move (\^{s}1$'$) consists in
  performing move (\^{s}1) at $v$ along $e',e''$ and subsequently
  performing move (\^{s}2) along the newly created edge.
\end{enumerate}

There is a natural correspondence between the moves (s$\star$) and
(\^{s}$\star$) described in the following

\begin{oss}
Let $H$ be a solid torus, and let $\delta$ be a triangulation of
$T=\partial H$. Let $\gamma$ be the simplex added by a move
(\^{s}$\star$) applied to $H$, and let $\delta'$ be the induced
triangulation of the surface $T'$, where $T'$ is the connected
component of $\partial U(T\cup\gamma)$ that lies inside $H$. Then
$\delta'$ is obtained from $\delta$ by a move (s$\star$).
\end{oss}

We can now establish the following

\begin{lemma}\label{trn:tor:2:tor:lem}
Let $\delta$ be a triangulation of a two-dimensional torus
$T=\partial H$, where $H$ is a solid torus. Then there exists a
triangulation $\tau$ of $H$ such that $\tau|_{\partial H}=\delta$,
and $\tau^{(0)}=\delta^{(0)}$.
\end{lemma}

\begin{proof}
We construct the desired triangulation by following the proof of
Lemma~\ref{trn:tor:2:tor:prelim:lem} and performing for each move
(s$\star$) the corresponding move (\^{s}$\star$). Here the surface
that tells us which move to apply, is, at each step, the connected
component of the boundary of $U(T\cup\{\gamma_i\})$ (where
$\gamma_i$ are all the simplices added by the moves (\^{s}$\star$)
already performed) that lies in the interior of $H$ and is endowed
with the obvious triangulation, whereas the simplices themselves are
glued, inside $H$, to the simplicial complex $T\cup\{\gamma_i\}$. By
Lemma~\ref{trn:tor:2:tor:prelim:lem} we will end up with a solid
torus $H'\subset H$ (possibly with immersed boundary) such that
$\partial H'$ is triangulated into two triangles (while
$H\setminus\Int H'$ is already triangulated in a desired fashion,
\emph{i.e.} such that all the vertices are in $\partial H$), and we
conclude the proof by adding a suitable layered triangulation.
\end{proof}

\paragraph{Proof of Theorem~\ref{int:protoMom:2:trn:thm}} Denote the
lateral tori of $(M,\Sigma,\Delta)$ by $T_1$, $\ldots$, $T_k$ and
consider the triangulations $\delta'(T_i)$ induced by $\Delta$, as
described at the beginning of this section. Let $H_i$ be the solid
torus bounded by $T_i$ in $N$. For each $i=1,\ldots,k$ denote by
$\tau_i$ the triangulation of $H_i$ provided by
Lemma~\ref{trn:tor:2:tor:lem} applied to $T_i=\partial H_i$ and
$\delta'(T_i)$. It now follows directly from the construction that
the triangulations $\tau_i$, $i=1,\ldots,k$, yield an ideal
triangulation $\tau$ of $N$ (we emphasize that the triangulation is
an ideal one because its vertices, obtained by contracting the lakes
of the $T_i$'s, all lie on $\Sigma$) and that the protoMom-structure
$(M,\Sigma,\Delta)$ is $\tau$-induced, whence the
conclusion.\finedimo

\section{Mom-subgraphs in $4$-valent graphs}\label{mom:subgraph:sec}

We will now describe a combinatorial tool that we will use to show
that any two internal protoMom-structures on a given $N$ are related
by the combinatorial moves that will be described in
Subsection~\ref{moves:Mom:subsec}. This tool is that of a
\emph{Mom-subgraph}; as will be explained in
Section~\ref{moves:main:sec} (see Lemma~\ref{max:e:mom:lem}) these
objects are dual, in a certain natural sense, to a particular type
of triangulation-induced protoMom-structures, the type that will
play an especially important role in establishing the rest of our
results.

\subsection{Minimal Mom-subgraphs}

The above-mentioned notion of a Mom-subgraph comes in two types, and
in this subsection we study the first of them, called a
\emph{minimal Mom-subgraph}.

\begin{defn}
Let $G$ be a connected $4$-valent graph. A \emph{minimal
Mom-subgraph} $\Gamma$ of $G$ is a complete coloring of the edges of
$G$ by colors $\{t,c,f\}$ such that:
\begin{itemize}
  \item the union $T(\Gamma)$ of the edges with color $t$ gives a maximal
  tree in $G$;
  \item precisely one edge, denoted $c(\Gamma)$, has color $c$.
\end{itemize}
\end{defn}

We also define $\hat{T}(\Gamma):=T(\Gamma)\cup c(\Gamma)$, and
$C(\Gamma)$ as the only simple closed curve contained in
$\hat{T}(\Gamma)$. Note that $c(\Gamma)\subset C(\Gamma)$.

\paragraph{Moves on minimal Mom-subgraphs}
We first describe two \emph{admissible moves} that can be applied to
a minimal Mom-subgraph $\Gamma$ of some $G$.
\begin{enumerate}
  \item[(m1)] if $e=c(\Gamma)$ and some edge $e'$ of color $f$ share a vertex $v$,
  switch the colors of $e$ and $e'$;
  \item[(m2)] if $e\subset T(\Gamma)$ and $e'\subset (G\setminus T(\Gamma))$
  share a vertex $v$, and $\left(T(\Gamma)\setminus e \right)\cup e'$ is
  connected, switch the colors of $e$ and $e'$.
\end{enumerate}
    \begin{figure}
    \begin{center}
    \includegraphics[scale=0.5]{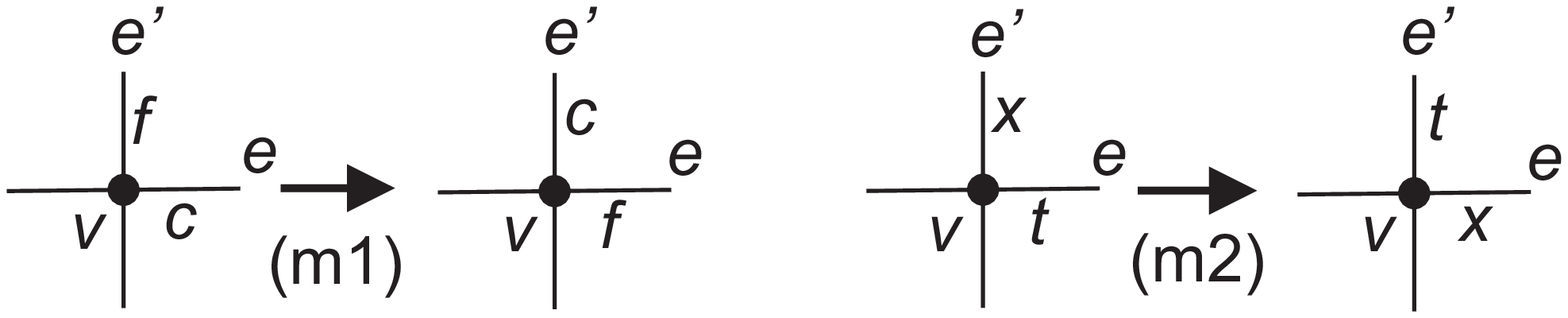}
    \mycap{Moves on Mom-subgraphs; on the right, we have $x\in\{c,f\}$.}
    \label{moves:m1:m2:fig}
    \end{center}
    \end{figure}
In both cases we will say that the move \emph{is applied at
$(v,e,e')$}, see Fig.~\ref{moves:m1:m2:fig}. The following statement
is obvious.

\begin{lemma}\label{m1m2:to:Mom:subgraph:lem}
The moves (m1) and (m2) transform a minimal Mom-subgraph into
another one.
\end{lemma}

We will mostly use not a single move (m2) but a particular
composition of several such moves, namely:
\begin{enumerate}
 \item[(\~{m}2)] if $e\subset T(\Gamma)$ and $e'\subset
 \left(G\setminus T(\Gamma)\right)$, and $e$ is contained in the (unique) simple path $\ell$ in
  $T(\Gamma)$ that joins the endpoints of $e'$, switch the colors
  of $e$ and $e'$.
\end{enumerate}
We will say that the move~(\~{m}2) \emph{is applied to $(e,e')$
along the cycle $C':=\ell\cup e'$}.

\begin{lemma}\label{m2:along:cycle:lem}
The move~(\~{m}2) is a composition of moves~(m2).
\end{lemma}

\begin{proof}
Let $\ell'$ be one of the connected components of
$\ell\setminus\Int(e)$, and let $e_1\subset\ell'$ be the edge
incident to $e'$. Denote the endpoints of $e'$ by $v'=e'\cap e_1$
and $v''$, and denote the other endpoint of $e_1$ by $v_1$. Observe
now that, since $e_1$ is separating for $T(\Gamma)$, the vertices
$v_1$ and $v'$ are contained in different connected components of
$T(\Gamma)\setminus\Int(e_1)$. Furthermore, the existence of the
path $\ell$ implies that $v''$ belongs to the same connected
component as $v_1$, hence we can apply the move (m2) at
$(v',e_1,e')$. The conclusion then follows by the induction on the
number of edges in the fixed $\ell'$.
\end{proof}

We will refer to the moves (m1) and (m2) as to \emph{admissible
moves}; by the above lemma, (\~{m}2) is a composition of such.

\paragraph{Relating minimal Mom-subgraphs by the moves} We now establish
the following result.

\begin{prop}\label{min:Mom:2:min:Mon:prop}
Let $G$ be a connected $4$-valent graph. Then any two minimal
Mom-subgraphs $\Gamma_1$ and $\Gamma_2$ of $G$ are connected by a
sequence of admissible moves.
\end{prop}

\begin{proof}
The proof is by induction on the number $n(\Gamma_1,\Gamma_2)$ of
edges $e$ of $G$ having different colors in $\Gamma_1$ and
$\Gamma_2$, with $e\subset\hat{T}(\Gamma_1)$ or
$e\subset\hat{T}(\Gamma_2)$. Note that, since $\hat{T}(\Gamma_1)$
and $\hat{T}(\Gamma_2)$ contain the same number of edges,
$n(\Gamma_1,\Gamma_2)$ cannot be equal to $1$. Furthermore, if
$n(\Gamma_1,\Gamma_2)=0$ then $\hat{T}(\Gamma_1)=\hat{T}(\Gamma_2)$,
therefore $C(\Gamma_1)=C(\Gamma_2)$, and the conclusion is obtained
by applying the move (\~{m}2) at the two $c$-edges along
$C(\Gamma_1)$.\vspace{4mm}

\noindent Suppose now that $n(\Gamma_1,\Gamma_2)>0$. We distinguish
two cases.\vspace{3mm}

\noindent (1) \emph{Assume that $T(\Gamma_1)\neq T(\Gamma_2)$.}
Since also $\hat{T}(\Gamma_1)\neq\hat{T}(\Gamma_2)$, we can find an
edge $e_2\subset\left(T(\Gamma_2)\setminus\hat{T}(\Gamma_1)\right)$.
We now take the only simple closed curve $\ell$ contained in
$T(\Gamma_1)\cup e_2$; since $T(\Gamma_2)$ is a tree, we can find
$e_1\subset\ell$ with $\Gamma_2(e_1)\neq t$. Then we apply to
$\Gamma_1$ the move (\~{m}2) at $(e_1,e_2)$ along $C=\ell\cup e_2$,
reducing $n(\Gamma_1,\Gamma_2)$. \vspace{3mm}

\noindent (2) \emph{Assume now that $T(\Gamma_1)=T(\Gamma_2)$.} This
immediately implies that $n(\Gamma_1,\Gamma_2)$ is $2$. More
precisely, for $e_i=c(\Gamma_i)$ with $i=1,2$ we have
$\Gamma_1(e_2)=f=\Gamma_2(e_1)$, while any other edge is assigned
the same color by both $\Gamma_1$ and $\Gamma_2$.\vspace{2mm}

\noindent The proof is by induction on the distance $k$ between
$C(\Gamma_1)$ and $C(\Gamma_2)$. Namely, let $\ell$ be the shortest
path in $T(\Gamma_1)$ joining $C(\Gamma_1)$ to $C(\Gamma_2)$; we
define $\ell$ to be empty if $C(\Gamma_1)\cap C(\Gamma_2)$ contains
at least one edge. Note that if $C(\Gamma_1)\cap C(\Gamma_2)$
consists of a single vertex $v$ then $\ell$ coincides with $v$. The
distance $k$ is defined to be the simplicial length of $\ell$,
\emph{i.e.} the number of edges in $\ell$; by definition, the length
of the empty path is $-\infty$. \vspace{2mm}

\noindent \emph{Base of induction:} $k\leqslant 0$. Then
$C(\Gamma_1)\cap C(\Gamma_2)$ is a non-empty path in
$T(\Gamma_1)=T(\Gamma_2)$; let $v$ be an end vertex of this path
(note that we might have $C(\Gamma_1)\cap C(\Gamma_2)=v$). Denote by
$e^{(1)}$ (one of) the extremal edge(s) of $C(\Gamma_1)\setminus
C(\Gamma_2)$ incident to $v$, and by $e^{(2)}$ (one of) the extremal
edge(s) of $C(\Gamma_2)\setminus C(\Gamma_1)$ incident to $v$. Apply
now to $\Gamma_1$ the following sequence of moves:
\begin{itemize}
  \item (\~{m}2) at $(e_1,e^{(1)})$ along $C(\Gamma_1)$;
  \item (\~{m2}) at $(e_2,e^{(2)})$ along $C(\Gamma_2)$;
  \item (m1) at $(v,e^{(1)},e^{(2)})$;
  \item (\~{m}2) at $(e_2,e^{(2)})$ along $C(\Gamma_2)$;
  \item (\~{m}2) at $(e_1,e^{(1)})$ along $C(\Gamma_1)$.
\end{itemize}
It is then easy to see that the resulting coloring is in fact
$\Gamma_2$. It follows from Lemma~\ref{m2:along:cycle:lem} that all
of the above operations are compositions of admissible moves, whence
the conclusion in this case.\vspace{2mm}

\noindent \emph{Inductive step.} Suppose that $k>0$; then $\ell$
contains at least one edge. Let $w$ be the vertex at which $\ell$
meets $C(\Gamma_1)$, and let $T'$ be the connected component of
$T(\Gamma_1)\setminus w$ containing $\ell$; notice that $T'$ is not
compact. We claim that there exists an $f$-edge $e$ with one
endpoint in $T'$ and the other not belonging to $T'$ (but possibly
coinciding with $w$). Indeed, assume that such an edge does not
exist. Then all $f$-edges with one endpoint on $T'$ would have the
other endpoint on $T'$ as well, so the obvious compactification of
$\left(T'\cup\{f-\mbox{edges with an endpoint on }T'\}\cup
e_2\right)$ would be a graph with one vertex of valence $1$ and all
the other vertices of valence $4$, which is impossible.

Consider now an edge $e$ with the above property, and let $\ell'$ be
the simple path in $T(\Gamma_1)$ joining its endpoints. Observe that
$\ell'$ contains $w$ and has at least one edge in common with
$\ell$. Denote by $e^{(1)}$ an edge of $C(\Gamma_1)$ incident to $w$
and not contained in $\ell'$, and by $e^{(2)}$ the edge of $\ell$
incident to $w$; notice that $e^{(2)}$ is contained in $\ell'$.
Apply now to $\Gamma_1$ the following moves:
\begin{itemize}
  \item (\~{m}2) at $(e_1,e^{(1)})$ along $C(\Gamma_1)$;
  \item (\~{m}2) at $(e,e^{(2)})$ along $\ell'\cup e$;
  \item (m1) at $(w,e^{(1)},e^{(2)})$;
  \item (\~{m}2) at $(e,e^{(2)})$ along $\ell'\cup e$;
  \item (\~{m}2) at $(e_1,e^{(1)})$ along $C(\Gamma_1)$.
\end{itemize}
Denote the resulting coloring by $\Gamma_1'$. Observe that
$\Gamma_1'$ and $\Gamma_2$ are still related as in case (2), and
that the distance between $C(\Gamma_1')$ and $C(\Gamma_2)$ is
strictly less than that between $C(\Gamma_1)$ and $C(\Gamma_2)$, and
this concludes the inductive step.
\end{proof}

\begin{rem}
\emph{In what follows (Proposition~\ref{max:same:trn:related:prop})
we will only use a weaker form of
Proposition~\ref{min:Mom:2:min:Mon:prop}, namely the fact that given
$\Gamma_1$, $\Gamma_2$ one can use (m1) and (m2) to transform
$\Gamma_1$ into some $\Gamma_1'$ such that
$\hat{T}(\Gamma_1')=\hat{T}(\Gamma_2)$, but for the sake of
completeness we present the above stronger version.}
\end{rem}

\subsection{General Mom-subgraphs}

In this subsection we consider a more general instance of
Mom-subgraphs than minimal ones:

\begin{defn}
Let $G$ be a connected $4$-valent graph. A \emph{general
Mom-subgraph} $\Gamma$ in $G$ is a complete coloring of the edges of
$G$ by colors $\{t,c,f\}$ such that:
\begin{itemize}
  \item the set $T(\Gamma)$ of the edges of $G$ that have color $t$
  spans $G$;
  \item each connected component of $T(\Gamma)$ is a tree;
  \item each edge colored by $c$ has both endpoints on the same
  connected component of $T(\Gamma)$, and for each connected
  component of $T(\Gamma)$ there is precisely one $c$-colored edge
  with vertices on it.
\end{itemize}
\end{defn}
\noindent Note that the set of $c$-edges is in bijective
correspondence with the set of connected components of $T(\Gamma)$.

For the rest of this subsection we fix the following notation. If
$\Gamma$ is a general Mom-subgraph of some $G$ then:
\begin{itemize}
  \item $k_{\Gamma}$ is the number of connected components of
  $T(\Gamma)$;
  \item $T_i(\Gamma)$ with $i=1,\ldots,k_{\Gamma}$ are the connected
  components of $T(\Gamma)$;
  \item $e_i$ is the $c$-edge with the endpoints on $T_i(\Gamma)$;
  \item $C_i(\Gamma)$ is the unique simple cycle in $T_i(\Gamma)\cup e_i$.
\end{itemize}

\paragraph{Moves on general Mom-subgraphs} The definition of the
moves (m1), (m2) and of the composite move (\~{m}2) naturally
extends to the context of general Mom-subgraphs. The extension is
\emph{verbatim} for (m1), whereas for (m2) and (\~{m}2) we add the
requirement that both endpoints of the edge $e'$ participating in
the move belong to the same component $T_i(\Gamma)$. In addition to
these, we will also consider three new moves:
\begin{enumerate}
  \item[(m2$'$)] let $e$ be an $f$-edge with one endpoint on
  $T_i(\Gamma)$ and the other on $T_j(\Gamma)\setminus C_j(\Gamma)$
  with $i\neq j$, and let $e'\subset T_j(\Gamma)\setminus C_j(\Gamma)$ be an edge
  incident to $e$ at a vertex $v$. Assume furthermore that $v$ and
  $C_j(\Gamma)$ belong to different connected components of
  $\left(T_j(\Gamma)\cup C_j(\Gamma)\right)\setminus\Int(e')$.
  Then switch the colors of $e$ and $e'$;
  \item[(m3)] let $e$ be an $f$-edge with one endpoint on
  $T_i(\Gamma)$ incident to some $e_j$ with $j\neq i$ (recall that $e_j$
  is a $c$-edge). Then assign color $t$ to $e$ and color $f$ to $e_j$;
  \item[(\={m}3)] let $e$ be an edge in $T_i(\Gamma)$, and let $e'$
  be an $f$-edge incident to $e$ and such that both endpoints of
  $e'$ belong to the same connected component of
  $T_i(\Gamma)\setminus\Int(e)$ and this component is not incident to the $c$-edge
  $e_i$. Then assign color $f$ to $e$ and
  color $c$ to $e'$.
\end{enumerate}
The moves are shown in Fig.~\ref{moves:m3:fig}. We will say that
(m2$'$) and (\={m}3) are applied at $(e,e')$ and that (m3) is
applied at $(e,e_j)$.
    \begin{figure}
    \begin{center}
    \includegraphics[scale=0.5]{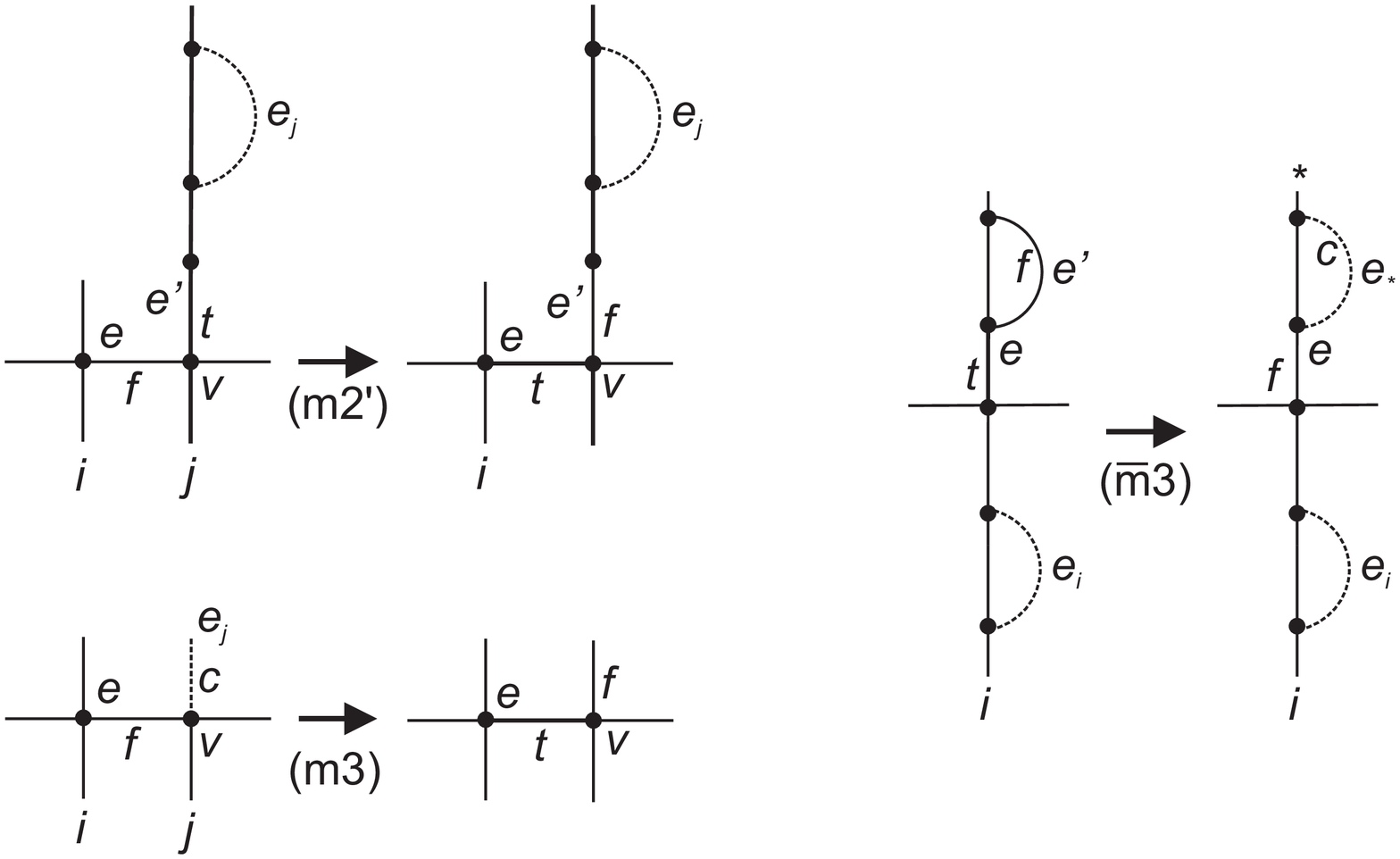}
    \mycap{The moves (m2$'$), (m3), and (\={m}3) on general Mom-subgraphs.}
    \label{moves:m3:fig}
    \end{center}
    \end{figure}
The next result is now readily established.

\begin{lemma}
We have the following:
\begin{enumerate}
  \item The moves (m2$'$), (m3), and (\={m}3) transform any general
  Mom-subgraph into a general Mom-subgraph.
  \item The inverse of a move (m2$'$) is a move (m2$'$).
  \item The inverse of a move (m3) is a move (\={m}3) and vice versa.
\end{enumerate}
\end{lemma}

For brevity, we will refer to the moves (m1), (m2), (m2'), (m3),
(\={m}3) as to \emph{admissible moves}.

\paragraph{Relating general Mom-subgraphs by the moves}
We now prove the following technical result.

\begin{prop}\label{gen:Mom:2:min:Mom:prop}
Let $G$ be a connected $4$-valent graph. Then every general
Mom-subgraph $\Gamma$ of $G$ can be transformed to a minimal
Mom-subgraph of $G$ via a sequence of admissible moves.
\end{prop}

\begin{proof}
We proceed by induction on $k_{\Gamma}$. The base of the induction
is $k_{\Gamma}=1$, in which case $\Gamma$ is already a minimal
Mom-subgraph, so there is nothing to prove.

Suppose that $k_{\Gamma}>1$. Since $G$ is connected, there exists an
edge of $G$ joining two distinct connected components of
$T(\Gamma)$, say $T_1(\Gamma)$ and $T_2(\Gamma)$. For each such edge
$e$ we can define $m_i(e)$ as the minimal edge-length of a path in
$T_i(\Gamma)$ connecting the vertex $e\cap T_i(\Gamma)$ to
$C_i(\Gamma)$; if $e$ is incident to $C_i(\Gamma)$ then the path
consists of one vertex only and its length is $0$. Let
$m(e)=\min\{m_1(e),m_2(e)\}$, and let $m_{\Gamma}=\min\{m(e)\}$,
where the minimum is taken over all edges connecting $T_1(\Gamma)$
to $T_2(\Gamma)$. We will show by induction on $m_{\Gamma}$ that
there is a sequence of admissible moves transforming $\Gamma$ to a
general Mom-subgraph $\Gamma'$ such that
$k_{\Gamma'}=k_{\Gamma}-1$.\vspace{3mm}

\noindent \emph{Base of induction:} $m_{\Gamma}=0$. This implies
that there exists an edge $e$ that has an endpoint on at least one
of $C_1(\Gamma)$, $C_2(\Gamma)$. Up to change of notation, we can
assume that $e$ has an endpoint on $C_2(\Gamma)$.

If $e$ is incident to $e_2$ then the desired $\Gamma'$ is obtained
by applying (m3) at $(e,e_2)$. Otherwise, denote by $e'$ an edge of
$C_2(\Gamma)$ incident to $e$. Then $\Gamma'$ is obtained by
applying first the move (\~{m}2) at $(e',e_2)$ along $C_2(\Gamma)$
and then applying (m3) at $(e,e')$.\vspace{3mm}

\noindent \emph{Inductive step.} Suppose that $m_{\Gamma}>0$, and
let $e$ be an $f$-edge realizing $m_{\Gamma}$. Up to change of
notation we can assume that $m_{\Gamma}=m_2(e)$. Set $v_i=e\cap
T_i(\Gamma)$.

Let $\ell$ be the shortest path in $T_2(\Gamma)$ joining $v_2$ to
$C_2(\Gamma)$, and let $e'$ be the edge of $\ell$ incident to $v_2$.
Apply the move (m2$'$) at $(e,e')$ and denote the resulting general
Mom-subgraph by $\Gamma_1$. Notice that $k_{\Gamma_1}=k_{\Gamma}$;
let $T_1(\Gamma_1)$ be the connected component of $T(\Gamma_1)$
containing $T_1(\Gamma)$, and let $T_2(\Gamma_1)$ be the component
containing $C_2(\Gamma)$.

Observe now that $T_1(\Gamma_1)$ and $T_2(\Gamma_1)$ are still
joined by an $f$-edge (\emph{e.g.} $e'$). Moreover, $\ell\setminus
e'$ joins $e'\cap T_2(\Gamma_1)$ to $C_2(\Gamma)=C_2(\Gamma_1)$.
Hence $m_{\Gamma_1}<m_{\Gamma}$, which concludes both inductions.
\end{proof}

We can now easily establish the main result of this section.

\begin{cor}\label{gen:Mom:2:gen:Mom:cor}
Let $G$ be a connected $4$-valent graph. Then any two general
Mom-subgraphs in $G$ are related by a sequence of moves (m1), (m2),
(m2'), (m3), (\={m}3).
\end{cor}

\section{Moves on protoMom-structures}\label{moves:main:sec}

In this section we will state and prove the main result of the
paper. In the first subsection we introduce two types of
combinatorial moves on weak protoMom-structures, and in the second
subsection we show that these moves are sufficient to relate to each
other any two weak protoMom-structures internal on the same
manifold.

\subsection{Description of the moves}\label{moves:Mom:subsec}

The moves that we describe can be broken down into two types, that
we call \emph{M-moves} and \emph{C-moves}. The meaning of the latter
is quite clear, as they correspond to so-called elementary collapses
and their inverses. The former moves are inspired by the procedure
of obtaining a protoMom-structure from an ideal triangulation (see
Section~\ref{trn:2:protomom:subsec}); we describe this origin below
in Remark~\ref{origin:M-move:rem}.

For all the moves considered we introduce a natural notion of
\emph{admissibility} by saying that \emph{a move is admissible if it
transforms a weak protoMom-structure into a weak
protoMom-structure.} It will be easy to see from the definition of
the moves that this is equivalent to requiring the moves to keep the
lateral boundary toral.

\paragraph{M-moves} The precise definition of an M-move is the
following one.

\begin{defn}\label{m-move:def}
Let $(M,\Sigma,\Delta)$ be a weak protoMom-structure internal on a
compact connected orientable $3$-manifold $N$ with $\partial
N=\Sigma$. Let $H$ be a $1$-handle of $\Delta$ of valence at least
$1$, and let $\alpha$ be a $2$-handle incident to $H$. Let
$\ell\subset\partial M$ be a closed embedded curve such that:
\begin{enumerate}
  \item[(a)] $\ell$ is disjoint from the $2$-handles of $\Delta$;
  \item[(b)] $\ell$ passes exactly $3$ times along $1$-handles,
  counting with multiplicity;
  \item[(c)] $\ell$ passes along $\partial H$, and the attaching curves of the $2$-handles
  different from $\alpha$ do not separate $\ell$ from the attaching curve of $\alpha$ on
  $\partial H$; and
  \item[(d)] $\ell$ bounds a disc in the complement of $M$ in $N$.
\end{enumerate}
Then the \emph{M-move at $(H,\alpha)$ along $\ell$} consists in
removing the $2$-handle $\alpha$ and gluing another $2$-handle along
$\ell$ instead.
\end{defn}
The actual shape of an M-move may vary depending on how the curve
$\ell$ and the boundary annulus of $\alpha$ pass along the various
$1$-handles (although the number of possibilities is obviously
limited). In Figg.~\ref{move:M:fig} and~\ref{move:M:bis:fig}
    \begin{figure}
    \begin{center}
    \includegraphics[scale=0.5]{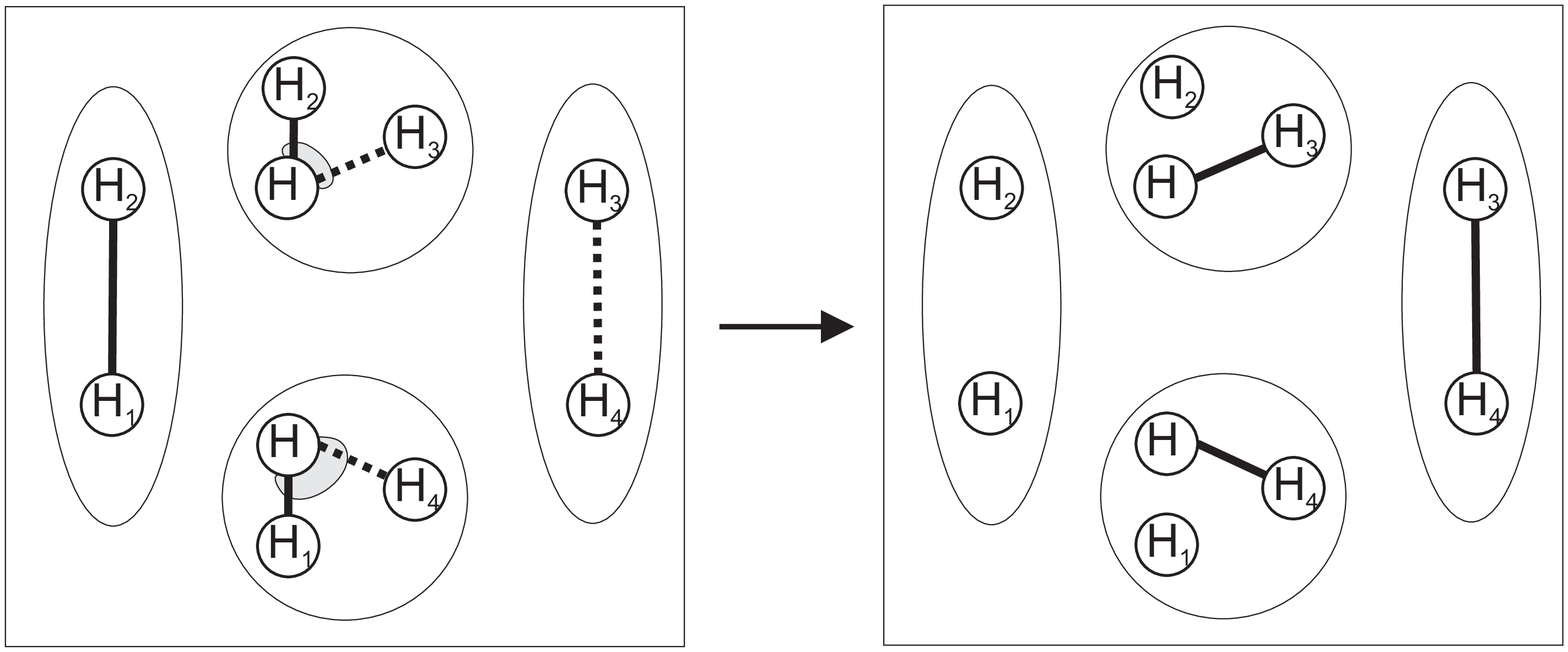}
    \mycap{An example of an M-move; on the left, the dotted line indicates segments of the curve $\ell$ and
    the thick line indicates segments of the core of the boundary annulus of $\alpha$. The areas shown in grey
    do not intersect any bridge other than that contained in $\alpha$, and the ovals indicate
    portions of lakes (which actually may or may not belong to the same lake). On the right the thick line
    indicates segments of the core of the boundary annulus of the $2$-handle replacing $\alpha$.}
    \label{move:M:fig}
    \end{center}
    \end{figure}
        \begin{figure}
    \begin{center}
    \includegraphics[scale=0.5]{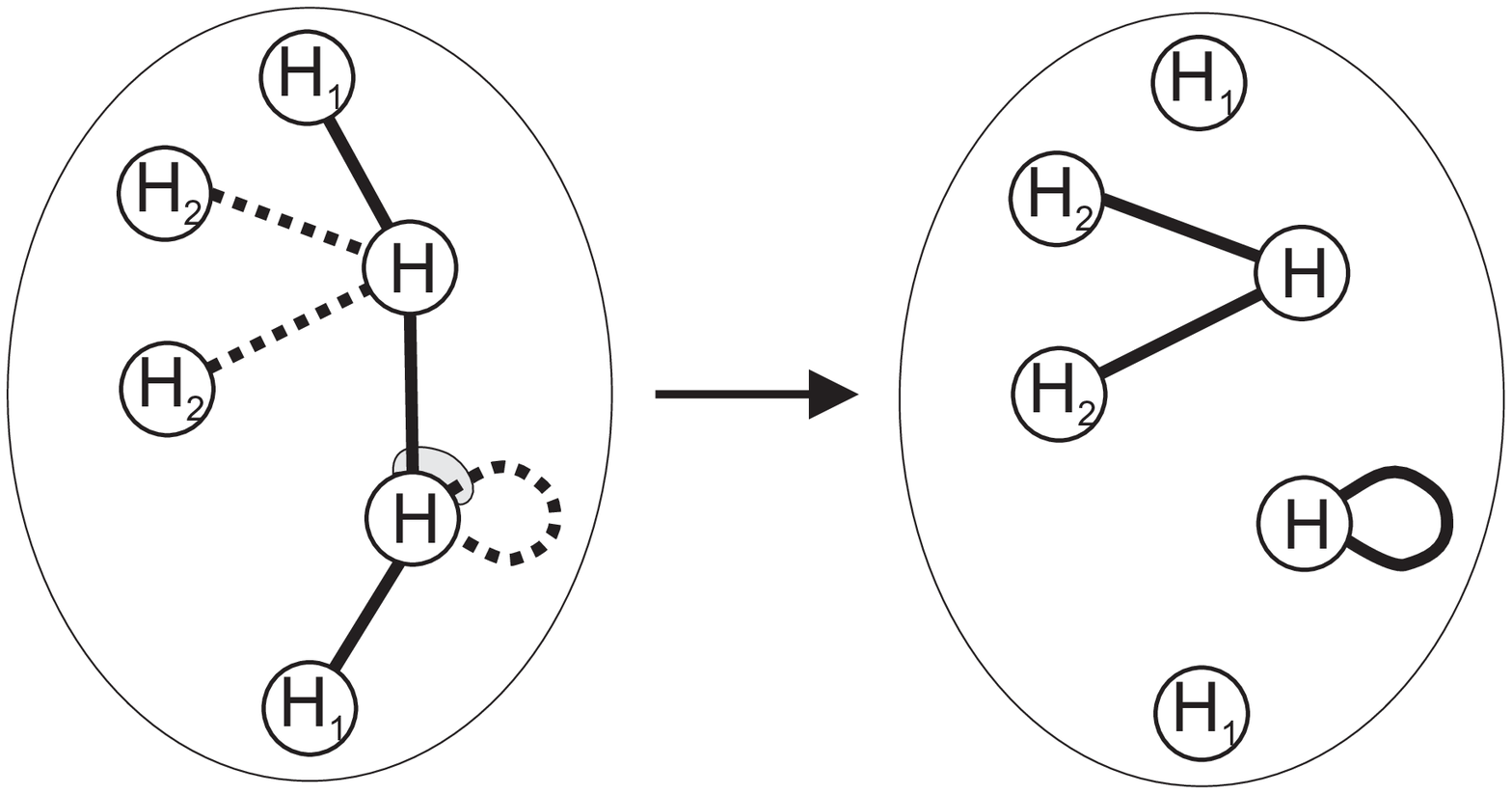}
    \mycap{Another example of an M-move; see the explanation for the previous figure. Note that
    in this case only one lake is involved in the move.}
    \label{move:M:bis:fig}
    \end{center}
    \end{figure}
we show two specific examples of an M-move.

\begin{rem}\label{origin:M-move:rem}
\emph{The requirements for the curve $\ell$ defining the move may
seem rather stringent, so at first glance the question of existence
of at least one curve with such properties may appear non-trivial.
Note however that the idea of the move comes from considering
triangulation-induced protoMom-structures, in particular those
obtained by removing some $2$-handles from the thickening of the
triangulation and keeping all $1$-handles; if we have such a
protoMom-structure then every tetrahedron such that one of its faces
is removed and another one is kept, provides a situation exactly as
described in the definition of an M-move.}
\end{rem}

\begin{rem}
\emph{An M-move is not automatically admissible. However, it follows
immediately from the definition that it is so if and only if after
its application the boundary still consists of tori.}
\end{rem}

\begin{rem}
\emph{It is clear from the definition that the (admissible) M-move
performed at $(H,\alpha)$ along $\ell$ is invertible with the
(admissible) inverse of the same type. More precisely, the inverse
of this move is the M-move performed at $(H,\alpha')$ along
$\ell_{\alpha}$, where $\alpha'$ is the $2$-handle inserted by the
initial move and $\ell_{\alpha}$ is the gluing curve of $\alpha$
(\emph{i.e.} the core circle of its boundary annulus).}
\end{rem}

\paragraph{C-moves} The second class of moves consists of
elementary collapses and their inverses, where by an
\emph{elementary collapse} we mean the removal of a $1$-handle of
valence $1$ together with the $2$-handle incident to it. The precise
description of these moves is as follows.

\begin{defn}
Let $(M,\Sigma,\Delta)$ be a weak protoMom-structure internal on a
compact connected orientable $3$-manifold $N$ with $\partial
N=\Sigma$, and let $\ell\subset\left(\partial
M\setminus\Sigma\right)$ be an embedded arc such that:
\begin{enumerate}
  \item[(a)] the endpoints of $\ell$ are contained in the union of
  the islands;
  \item[(b)] $\ell$ intersects the union of the $1$-handles along
  precisely two segments and the union of the lakes along precisely one
  segment;
  \item[(c)] $\ell$ is disjoint from the $2$-handles of $\Delta$.
\end{enumerate}
Then the \emph{C-move along $\ell$} consists in first adding a
$1$-handle $H$ contained in $N$ and parallel to $\partial M$ and
with bases in the immediate vicinity of the endpoints of $\ell$,
then completing $\ell$ to a closed curve $\ell'$ that passes
precisely once along $H$ and is disjoint from any $2$-handles, and
finally gluing a new $2$-handle along $\ell'$.
\end{defn}
The C-move has three distinct shapes that are shown in
Figg.~\ref{move:C1:fig},~\ref{move:C2:fig}, and~\ref{move:C3:fig}.
    \begin{figure}
    \begin{center}
    \includegraphics[scale=0.5]{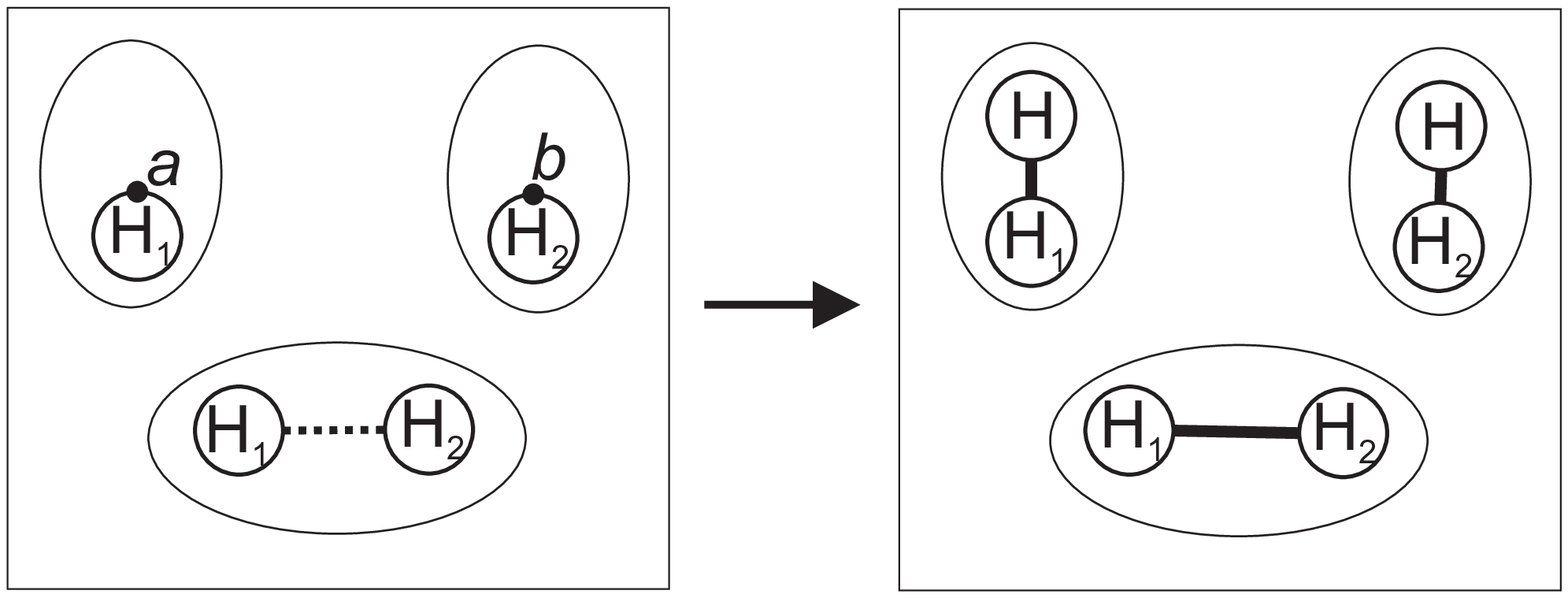}
    \mycap{The first type of the C-move; $a$ and $b$ denote the endpoints of $\ell$.}
    \label{move:C1:fig}
    \end{center}
    \end{figure}
    \begin{figure}
    \begin{center}
    \includegraphics[scale=0.5]{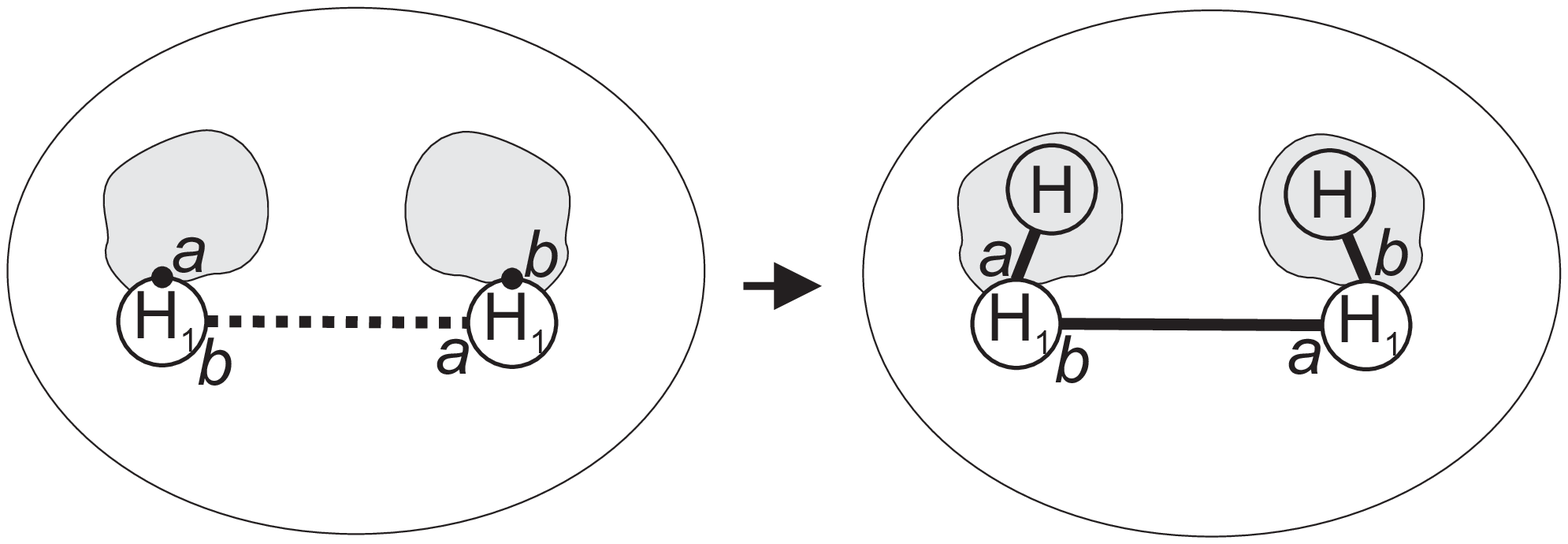}
    \mycap{The second type of the C-move; $a$ and $b$ indicate how the segment passes along the $1$-handle $H_1$.}
    \label{move:C2:fig}
    \end{center}
    \end{figure}
    \begin{figure}
    \begin{center}
    \includegraphics[scale=0.5]{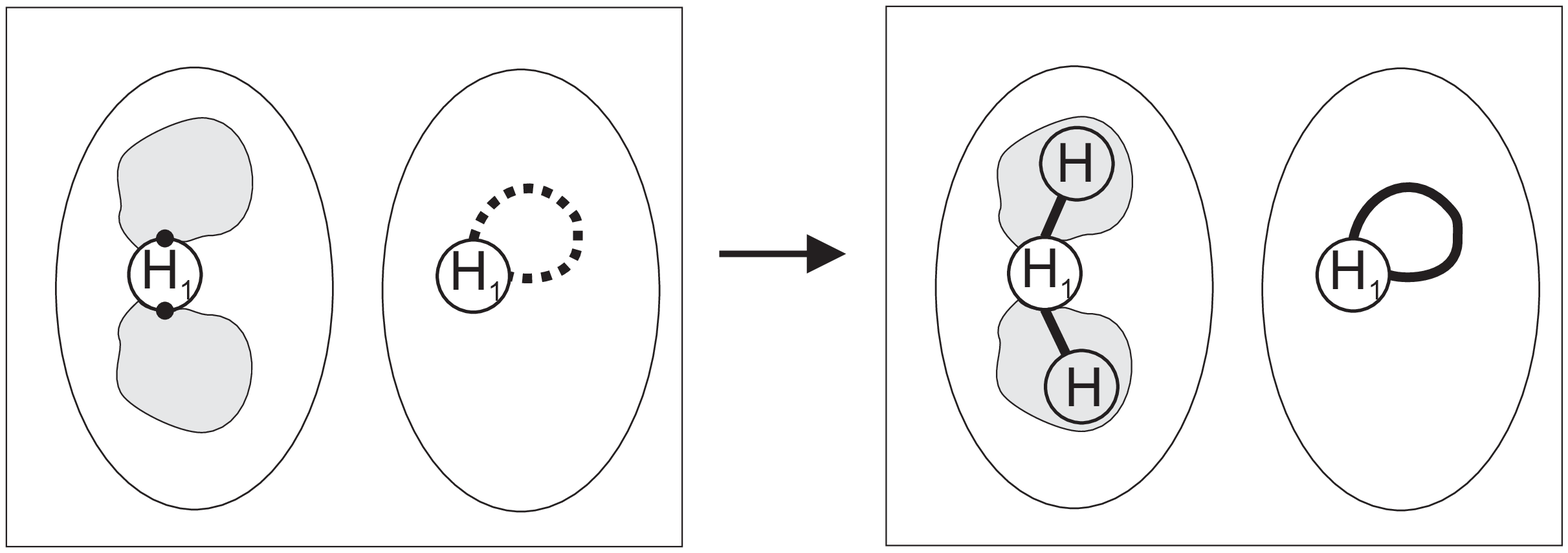}
    \mycap{The third type of the C-move.}
    \label{move:C3:fig}
    \end{center}
    \end{figure}

It is clear that any C-move is automatically admissible. It is also
invertible, with the inverse being a suitable elementary collapse,
and we call any of the latter a C$^{-1}$-move. We will typically
refer to all C$^{\pm 1}$-moves as simply to C-moves.

\subsection{Relating protoMom-structures by the moves}

We are now ready to state the main result of the paper.

\begin{thm}\label{main:thm}
Let $N$ be a compact connected orientable $3$-manifold such that
$\partial N$ is a non-empty surface without spherical components.
Then any two full weak protoMom-structures internal on $N$ are
related by a sequence of admissible M- and C-moves.
\end{thm}

The proof of this result is based on the statement of
Theorem~\ref{int:protoMom:2:trn:thm} and on the construction used to
prove it, and on Corollary~\ref{gen:Mom:2:gen:Mom:cor}. We start by
introducing some additional notions and proving some easy facts. For
the rest of the section we fix an ambient manifold $N$ as in
Theorem~\ref{main:thm}.

\paragraph{General Mom-subgraphs and protoMom-structures} Let $\tau$
be an ideal triangulation of $N$; then, as described in
Section~\ref{prelim:sec}, dual to $\tau$ there is a special spine
$P_{\tau}$ of $N$. Note that any general Mom-subgraph $\Gamma$ on
$S(P_{\tau})$, which is a $4$-valent graph, defines a weak
protoMom-structure $(M_{\Gamma},\Sigma,\Delta_{\Gamma})$ internal on
$N$: the $1$-handles of $\Delta_{\Gamma}$ are thickenings of the
edges of $\tau$, and the $2$-handles of $\Delta_{\Gamma}$ are
thickenings of the faces of $\tau$ dual to the edges of
$S(P_{\tau})$ having color $f$ in $\Gamma$. We will call
$(M_{\Gamma},\Sigma,\Delta_{\Gamma})$ the protoMom-structure
\emph{dual} to $\Gamma$. This duality extends to the moves:

\begin{prop}\label{moves:graph:and:str:prop}
Let $\tau$, $\Gamma$, and $(M_{\Gamma},\Sigma,\Delta_{\Gamma})$ be
as above, and let $\Gamma'$ be obtained from $\Gamma$ by an
admissible move of type (m1), (m2), (m2'), (m3), or (\={m}3). Then:
\begin{enumerate}
  \item If the move is of type (m1), (m2'), (m3), or (\={m}3)
  then the dual structure $(M_{\Gamma'},\Sigma,\Delta_{\Gamma'})$ is obtained from
  $(M_{\Gamma},\Sigma,\Delta_{\Gamma})$ by an admissible M-move.
  \item If the move is of type (m2) then either
  $(M_{\Gamma'},\Sigma,\Delta_{\Gamma'})$ coincides with
  $(M_{\Gamma},\Sigma,\Delta_{\Gamma})$ or it is obtained from
  $(M_{\Gamma},\Sigma,\Delta_{\Gamma})$ by an admissible M-move.
\end{enumerate}
\end{prop}

\paragraph{Weak protoMom-structures maximal with respect to a
triangulation} We now present an alternative way to describe the set
of protoMom-structures obtained by taking the duals of all
Mom-subgraphs in $S(P_{\tau})$, where $\tau$ is a fixed ideal
triangulation of $N$.

\begin{defn}\label{trn:max:str:def}
We say that a $\tau$-induced weak protoMom-structure
$(M,\Sigma,\Delta)$ is \emph{$\tau$-maximal} if it is not properly
contained in any other $\tau$-induced protoMom-structure on $N$.
\end{defn}

\begin{rem}
\emph{The above definition is equivalent to requiring the set of the
$1$-handles of $(M,\Sigma,\Delta)$ to coincide with the set of
thickenings of \emph{all} the edges of $\tau$.}
\end{rem}

\begin{rem}
\emph{When speaking about $\tau$-maximal protoMom-structures, we
implicitly fix for each edge or triangle of $\tau$ a specific
thickening of it, used then for all $\tau$-induced
protoMom-structures.}
\end{rem}

\begin{lemma}\label{max:e:mom:lem}
Let $\tau$ be an ideal triangulation of $N$. Then the set of
$\tau$-maximal weak protoMom structures on $N$ is precisely the set
of weak protoMom-structures dual to general Mom-subgraphs of
$S(P_{\tau})$.
\end{lemma}

\begin{proof}
A weak protoMom-structure dual to a Mom-subgraph is $\tau$-maximal
because it contains all $1$-handles, and insertion of some
$2$-handle would create sphere boundary components. Let now
$(M,\Sigma,\Delta)$ be a $\tau$-maximal weak protoMom-structure.
Consider the auxiliary coloring $\Gamma'$ of $S(P_{\tau})$ obtained
by assigning color $f$ to the edges dual to the $2$-handles that are
contained in $\Delta$, and a new color $x$ to all the other edges.
Since $\partial M$ has no spherical components, no connected
component of $x$-colored edges is contractible; furthermore,
$\partial M\setminus\Sigma$ consists of tori, hence we can find a
general Mom-subgraph $\Gamma$ such that the set of its $t$-colored
edges and $c$-colored edges is contained in the set of $x$-colored
edges (relative to $\Gamma'$). Then it is clear that
$(M,\Sigma,\Delta)$ is contained in
$(M_{\Gamma},\Sigma,\Delta_{\Gamma})$, and since $(M,\Sigma,\Delta)$
is $\tau$-maximal, they must actually coincide.
\end{proof}

Corollary~\ref{gen:Mom:2:gen:Mom:cor} and
Proposition~\ref{moves:graph:and:str:prop} now imply:

\begin{prop}\label{max:same:trn:related:prop}
Let $\tau$ be an ideal triangulation of $N$. Then any two
$\tau$-maximal weak protoMom-structures are related by a sequence of
admissible M-moves.
\end{prop}

\paragraph{Recovering a maximal protoMom-structure from an arbitrary
one} The last essential tool for the proof of Theorem~\ref{main:thm}
is the following:

\begin{lemma}\label{proto:2:max:lem}
Let $(M,\Sigma,\Delta)$ be a full weak protoMom-structure internal
on $N$, and let $\tau$ be the triangulation constructed in the proof
of Theorem~\ref{int:protoMom:2:trn:thm}. Then there exists a
sequence of C-moves transforming $(M,\Sigma,\Delta)$ into a
$\tau$-maximal weak protoMom structure.
\end{lemma}

\begin{proof}
Recall from the proof of Theorem~\ref{int:protoMom:2:trn:thm} that
$\tau$ can be explicitly constructed by compressing each handle of
$\Delta$ on its core and performing a sequence of moves (\^{s}1),
(\^{s}2), (\^{s}3) described in Section~\ref{mom:2:trn:subsec}
(whose definition would be adjusted in a straightforward manner to
account for the fact that we do not compress lakes and the triangles
and tetrahedra that we insert are actually truncated). Indeed,
recall that the move (\^{s}1$'$) is actually a composition of
(\^{s}1) and (\^{s}2) and that the layered triangulation inserted at
the conclusion can be obtained by a sequence of (\^{s}2)-moves
followed by filling in the one-tetrahedron triangulation of solid
torus.

Let us fix such a sequence of (\^{s}$\star$)-moves. To each move we
now associate either a C-move or an identity (empty) move on the
protoMom-structure $(M,\Sigma,\Delta)$. More precisely, we do the
following:
\begin{itemize}
  \item to each (\^{s}1)-move we associate the C-move that consists
  in insertion of the 2-handle and the 1-handle that are thickenings
  of respectively the triangle and the edge inserted by the
  (\^{s}1)-move under consideration;
  \item to each (\^{s}2)-move we associate the
  C-move that consists  in insertion of a precisely one 2-handle
  and one 1-handle where the 1-handle is the thickening of the edge
  inserted during the (\^{s}2)-move under consideration and the 2-handle
  is the thickening of precisely one of the two triangles inserted
  by the move (this association is therefore non-unique);
  \item no move is associated to any of the (\^{s}3)-moves.
\end{itemize}
(Note that we do not associate anything to the final operation of
filling in with the one-tetrahedron triangulation). Then the
sequence of (\^{s}$\star$)-moves fixed above yields a well-defined
sequence of C-moves on $(M,\Sigma,\Delta)$. Applying this sequence
of moves to $(M,\Sigma,\Delta)$ gives a new weak internal
protoMom-structure $(M',\Sigma',\Delta')$ on $N$. Moreover,
$(M',\Sigma',\Delta')$ is also $\tau$-induced and the set of its
1-handles coincides with the set of thickenings of all the edges of
$\tau$. It now follows from the Euler characteristic argument that
$(M',\Sigma',\Delta')$ is $\tau$-maximal, whence the conclusion.
\end{proof}

\paragraph{Proof of Theorem~\ref{main:thm}} By Proposition~\ref{max:same:trn:related:prop} and
Lemma~\ref{proto:2:max:lem} the admissible M- and C-moves are
sufficient to relate to each other any two full weak
protoMom-structures induced by the same ideal triangulation $\tau$
of $N$. Recall now \cite{Ma88,Pierga} that two ideal triangulations
$\tau_1$ and $\tau_2$ of $N$ are related to each other by a sequence
of instances of the $(2\rightarrow 3)$-move shown in
Fig.~\ref{2:to:3:move:fig} and of inverses of this move.
    \begin{figure}
    \begin{center}
    \includegraphics[scale=0.5]{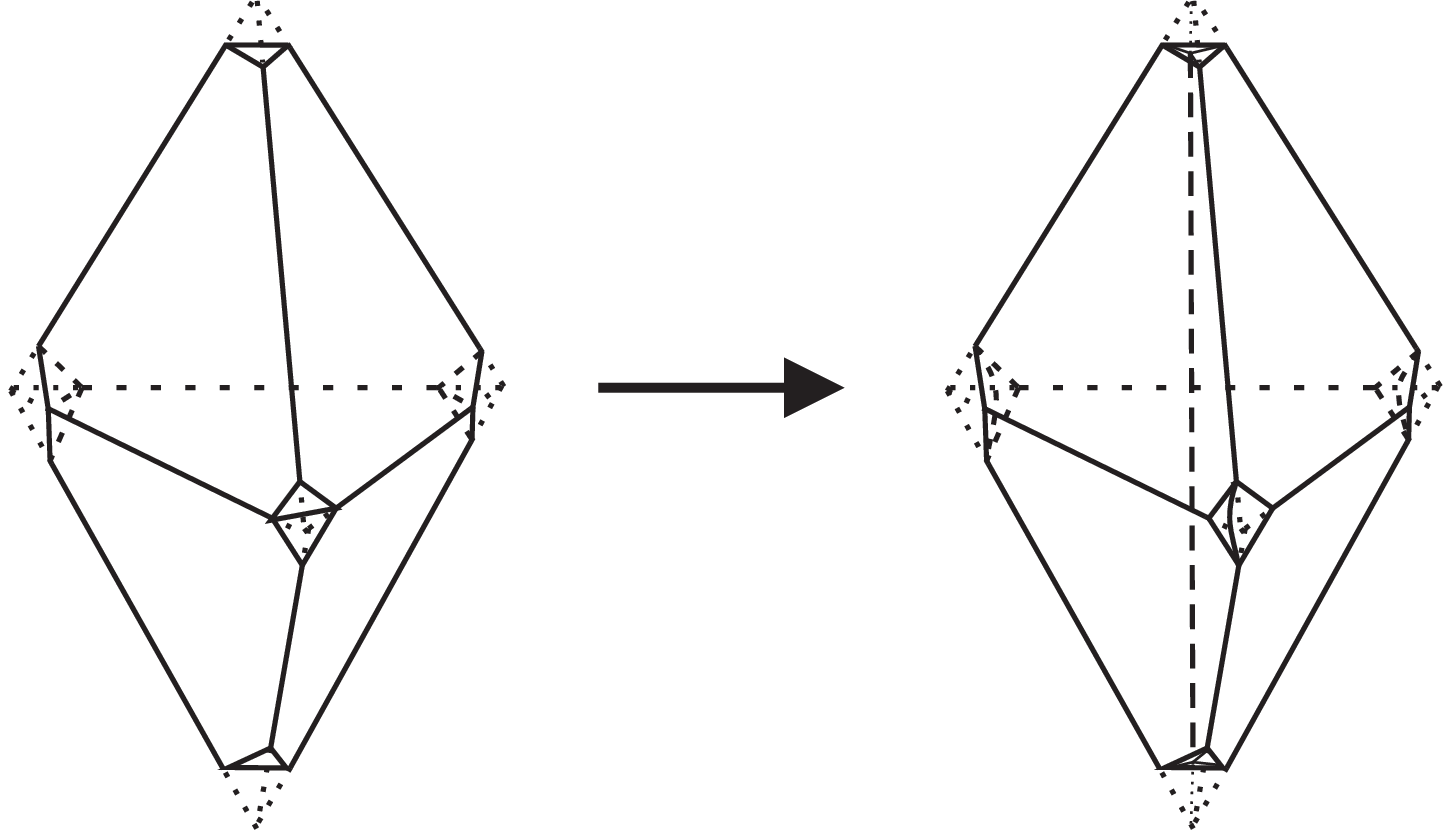}
    \mycap{The $(2\rightarrow 3)$-move on a truncated triangulation.}
    \label{2:to:3:move:fig}
    \end{center}
    \end{figure}
(This would be false if $\tau_1$ or $\tau_2$ were to consist of only
one tetrahedron, but this is impossible for an orientable $N$
without spherical boundary components.)

It is then sufficient to show that if $\tau_2$ is obtained from
$\tau_1$ by one $(2\rightarrow 3)$-move then there exist a
$\tau_1$-maximal weak protoMom-structure and a $\tau_2$-maximal weak
protoMom-structure that are related by M- and C-moves. Indeed, let
$\alpha$ be the $2$-handle that thickens the triangle destroyed by
the $(2\rightarrow 3)$-move, and let $(M_1',\Sigma,\Delta_1')$ be
any $\tau_1$-maximal protoMom-structure not containing $\alpha$; its
existence is obvious. Now we note that $(M_1',\Sigma,\Delta_1')$ can
actually be viewed as a $\tau_2$-induced protoMom-structure, and
there exists a C-move transforming it into a $\tau_2$-maximal
structure $(M_2',\Sigma,\Delta_2')$; this move consists in inserting
the edge and exactly one of the three triangles that appear during
the $(2\rightarrow 3)$-move and then thickening them. The proof is
whence complete.~\finedimo

\begin{rem}
\emph{By Theorem~\ref{main:thm} two genuine (non-weak)
Mom-structures internal on the same cusped hyperbolic manifold are
related by a sequence of admissible M- and C-moves. However, even
starting with two Mom-structures for which there exists a
triangulation with respect to which both structures are maximal,
along the sequence one might very well encounter non-genuine (weak
in our sense) Mom-structures. Therefore the question of finding
combinatorial moves relating to each other any two genuine
Mom-structures remains open.}
\end{rem}

\vspace{1cm}

\noindent Dipartimento di Matematica\\
Largo Bruno Pontecorvo 5\\
56127 PISA -- Italy\\
{\tt pervova@guest.dma.unipi.it}\\

\end{document}